\documentclass[11pt]{article}
\usepackage{mathrsfs}
\usepackage{amsfonts}
\usepackage{amssymb,amsmath,amsthm}
\usepackage{graphicx}
\usepackage{subfigure}
\usepackage{color}
\usepackage{float}

\def\disp{\displaystyle}

\def\dref#1{(\ref{#1})}
\def\crr{\cr\noalign{\vskip2mm}}
\def\dfrac{\displaystyle\frac}

\numberwithin{equation}{section}

\newtheorem{theorem}{Theorem}[section]

\newtheorem{proposition}{Proposition}[section]
\newtheorem{lemma}{Lemma}[section]
\newtheorem{remark}{Remark}[section]

\newcommand{\Ascr} {{\cal A}}
\newcommand{\Bscr} {{\cal B}}

\newcommand{\Dscr} {{\cal D}}
\newcommand{\Hscr} {{\cal H}}
\newcommand{\Lscr} {{\cal L}}
\newcommand{\Nscr} {{\cal N}}
\newcommand{\Xscr} {{\cal X}}
\newcommand{\aline}  {{\mathbb A}}
\newcommand{\bline}  {{\mathbb B}}
\renewcommand{\cline}{{\mathbb C}}
\newcommand{\fline}  {{\mathbb F}}
\renewcommand{\hline}{{\mathbb H}}

\newcommand{\pline}  {{\mathbb P}}
\newcommand{\rline}  {{\mathbb R}}
\newcommand{\GGG}    {{\mathbf G}}
\newcommand{\rarrow} {\mathop{\rightarrow}}
\newcommand{\half}   {{\frac{1}{2}}}
\newcommand{\e}      {{\varepsilon}}

\renewcommand{\l}    {{\lambda}}
\renewcommand{\Re}{{\rm Re\,}}

\newcommand{\dd}  {{\rm d}\hbox{\hskip 0.5pt}}
\newcommand{\FORALL} {{\hbox{$\hskip 11mm \forall \;$}}}
\newcommand{\mm}     {{\hbox{\hskip 0.5pt}}}
\newcommand{\m}      {{\hbox{\hskip 1pt}}}
\newcommand{\nm}     {{\hbox{\hskip -3pt}}}
\newcommand{\bluff}  {{\hbox{\raise 15pt \hbox{\mm}}}}
\newcommand{\sbluff} {{\hbox{\raise  9pt \hbox{\mm}}}}
\newcommand{\bbm}[1]{\left[\begin{matrix} #1 \end{matrix}\right]}
\newcommand{\sbm}[1]{\left[\begin{smallmatrix} #1
   \end{smallmatrix}\right]}

\begin{document}
\renewcommand{\thefootnote}{\fnsymbol{footnote}}
\renewcommand{\thefootnote}{\fnsymbol{footnote}}
\newcommand{\footremember}[2]{%
   \footnote{#2}
    \newcounter{#1}
    \setcounter{#1}{\value{footnote}}%
}
\newcommand{\footrecall}[1]{%
    \footnotemark[\value{#1}]%
}
\makeatletter
\def\blfootnote{\gdef\@thefnmark{}\@footnotetext}
\makeatother

\begin{center}
{\LARGE \bf Output feedback exponential stabilization for 1-D \\ 
unstable wave equations with boundary control \\[0.6ex]
 matched disturbance
}\\[4ex]
Hua-Cheng Zhou,  \ \ \ George Weiss\;
\blfootnote{This work was supported by the Israel Science
Foundation under grant 800/14.}
\blfootnote{H.-C. Zhou (hczhou@amss.ac.cn) and G. Weiss
        (gweiss@eng.tau.ac.il) are with the School of Electrical
        Engineering, Tel Aviv University, Ramat Aviv, Israel,
        69978.}
\end{center}
\vspace{3mm}
%
%

\noindent {\bf Abstract:} We study the output feedback exponential
stabilization of a one-dimensional unstable wave equation, where the
boundary input, given by the Neumann trace at one end of the domain,
is the sum of the control input and the total disturbance. The latter
is composed of a nonlinear uncertain feedback term and an external
bounded disturbance. Using the two boundary displacements as output
signals, we design a disturbance estimator that does not use high
gain. It is shown that the disturbance estimator can estimate the
total disturbance in the sense that the estimation error signal is in
$L^2[0,\infty)$. Using the estimated total disturbance, we design an
observer whose state is exponentially convergent to the state of
original system. Finally, we design an observer-based output feedback
stabilizing controller. The total disturbance is approximately
canceled in the feedback loop by its estimate. The closed-loop system
is shown to be exponentially stable while guaranteeing that all the
internal signals are uniformly bounded. \vspace{3mm}

\vspace{3mm}


\section{Introduction}

In this paper, we are concerned with the following one-dimensional
wave equation:
\begin{equation} \label{wave-o}
   \left\{\begin{array}{rl} w_{tt}(x,t) \nm &=\m w_{xx}(x,t),\crr
   \disp w_x(0,t) \nm &=\m -qw(0,t),\crr\disp w_x(1,t) &=\m u(t)+
   f(w(\cdot,t), w_t(\cdot,t))+d(t),\crr\disp w(x,0) \nm &=\m w_0(x),
   \ \ \ w_t(x,0) \m=\m w_1(x),\crr\disp y_{m}(t) &=\m (w(0,t)\m,
   w(1,t)), \end{array} \right.
\end{equation}
where $x\in(0,1)$, $t\geq 0$, $(w,w_t)$ is the state, $u$ is the
control input signal, and $y_m$ is the output signal, that is, the
boundary traces $w(0,t)$ and $w(1,t)$ are measured. The equation
containing the constant $q>0$ creates a destabilizing boundary
feedback at $x=0$ that acts like a spring with negative spring
constant. $f:H^1(0,1)\times L^2(0,1)\to\mathbb{R}$ is an unknown
possibly nonlinear mapping that represents the {\em internal
uncertainty} in the model, and $d$ represents the unknown {\em
external disturbance}, which is only supposed to satisfy $d\in
L^\infty[0,\infty)$. For the sake of simplicity, we denote
\begin{equation} \label{Jared}
   F(t) :=\m f(w(\cdot,t),w_t(\cdot,t)) + d(t)
\end{equation}
and we call this signal the {\em total disturbance}. We often write
$\dot w$ instead of $w_t$.

\medskip
\begin{figure}[H]\centering
\centering{}\unitlength 0.6mm
\begin{picture}(170,35)(0,55)
\put(17,83){\line(1,0){136}} 
\put(25,90){\line(0,-1){12}} 
\put(25,78){\line(-5,-1){6}} 
\put(19,76.8){\line(5,-1){12}} 
\put(156,81.5){zero level}
\put(31,74.4){\line(-5,-1){12}} 
\put(19,71.9){\line(5,-1){12}}
\put(31,69.5){\line(-5,-1){12}}
\put(19,67.1){\line(5,-1){12}}
\put(31,64.7){\line(-5,-1){12}}
\put(19,62.3){\line(5,-1){12}}
\put(31,59.8){\line(-4,-1){6}}
\put(25,58.5){\line(0,-1){5}} 
\put(-4,57){negative}
\put(34,57){spring}
\put(17,53.5){\line(1,0){17}} 
\put(21,53.4){\line(-2,-1){6}} 
\put(24,53.4){\line(-2,-1){6}}
\put(27,53.4){\line(-2,-1){6}}
\put(30,53.4){\line(-2,-1){6}}
\put(33,53.4){\line(-2,-1){6}} 
\put(33,68){$q$}
\put(23.5,89.5){$\circ$}
\put(25,91){\line(6,-1){6}} 
\put(33,89){\line(5,-2){6}}
\put(42,85){\line(5,-3){6}} 
\put(50,80){\line(5,-4){6}}
\put(58,74){\line(5,-2){6}}
\put(65,71){\line(1,0){4}}
\put(71,71){\line(5,2){6}}
\put(78,74){\line(5,3){6}}
\put(86,79){\line(5,3){6}}
\put(94,84){\line(5,2){6}}
\put(102,87){\line(5,1){6}}
\put(109.5,88.6){\line(5,0){4}} 
\put(115,88){\line(5,-2){6}}
\put(122,85){\line(5,-3){6}}
\put(129,80){\line(5,-4){6}}
\put(136,74){\line(3,-2){8}} 
\put(142.5,67.2){$\circ$}
{\color{red}
\put(144.3,55){\vector(0,1){12.1}} 
\put(144.15,55){\vector(0,1){12.1}} 
\put(144,55){\vector(0,1){12.1}} 
\put(147,55){$F=f+d$}
\put(144.3,89){\vector(0,-1){18.2}} 
\put(144.15,89){\vector(0,-1){18.2}} 
\put(144,89){\vector(0,-1){18.2}} 
\put(132,91){control force $u$}}
\put(118.5,60){\vector(3,1){21}} 
\put(75,58){\vector(-3,2){45}} 
\put(81,59){dispacement}
\put(78,53){measurements}
\put(75,91){string}
\put(93,91){\vector(2,-1){6}} 
\end{picture} \vspace{2mm}
\caption{\label{un-wave-pic} Our plant, an unstable string system}
\end{figure}
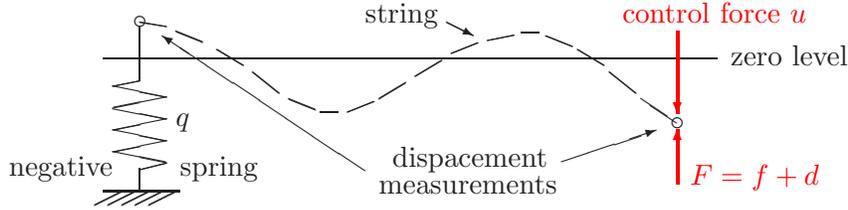 \vspace{-2mm}

We consider system \dref{wave-o} in the state Hilbert space
$\hline=H^1(0,1)\times L^2(0,1)$ with the inner product given by
\vspace{-2mm}
\begin{equation} \label{Eliel}
   \langle(\phi_1,\psi_1),(\phi_2,\psi_2)\rangle_\hline \m=
   \int_0^1[\phi_1'(x)\overline{\phi_2'(x)}+\psi_1(x)\overline{\psi_2
   (x)}]\dd x +\phi_1(0)\overline{\phi_2(0)} \m.
\end{equation}

The objective of this paper is to design a feedback controller which
generates the control signal $u$, using only the measurements $y_m$,
such that the state of the closed-loop system (that includes the state
of the system \dref{wave-o}) converges to zero, exponentially. Later
in the paper we shall also discuss a related problem, where the
negative spring is replaced by a negative damper. More precisely, on
the right hand-side of the equation in \dref{wave-o}) containing $q$,
we have $-qw_t(0,t)$ (instead of $-qw(0,t)$). We shall solve the
exponential stabilization problem also for this alternative nonlinear
wave system \dref{wave-o-US-N}. These results have been announced
(without proof) in the IFAC conference paper
\cite{ZhouWeiss_IFAC:2017}.

For simplicity of implementation, it is desirable to use a small
number of input and output signals for output feedback stabilization.
For the disturbance free situation (that is, $f\equiv 0$ and $d\equiv
0$), the stabilization of the system \dref{wave-o} was first
investigated in \cite{KGBS:08}, who used two measurement signals to
obtain an exponentially stable closed-loop system. Using only one
displacement signal as measurement, strong stability of the closed
loop system was achieved in \cite{GuoGuo2009auto}, using Lyapunov
functionals. In the recent paper \cite{F-Gtac}, the output signal is
only one displacement signal and an exponentially stabilizing
controller is designed by using a new ``backstepping'' method.
However, when the total disturbance $F$ acts at the control end, the
stabilization problem for \dref{wave-o} becomes much more difficult.
Here we present a dynamic compensator which employs a disturbance
estimator described by partial differential equations (PDEs) and full
state feedback based on the observer state. Our compensator consists
of two parts: the first part is to cancel the total disturbance by
applying the active disturbance rejection control (ADRC) strategy,
which is an unconventional design strategy first proposed by Han in
1998 \cite{Han1998}; the second part is to stabilize the system by
using the classic backstepping approach. The stabilization problem of
system \dref{wave-o} has been considered first in \cite{GWGBZ2013tac},
where the vector of output measurement was taken to be $y_m(t)=(w(0,t)
,w_t (1,t))$ and the disturbance has the following form: \vspace{-2mm}
$$ d(t) \m=\m \sum_{j=1}^m [\theta_j\sin\alpha_jt+\vartheta_j\cos
   \alpha_jt],\ \ \ t\geq 0 \m,$$
with known frequencies $\alpha_j$ and unknown amplitudes $\theta_j,
\vartheta_j,\ j=1,2,\ldots m$, and the resulting closed-loop system is
asymptotically stable. Obviously, the disturbance signal in this paper
is more general than the one described above. Recently, the
stabilization problem of system \dref{wave-o} with $f\equiv0$, $d\in
L^\infty[0,\infty)$ has been investigated in \cite{FG2016}, where the
output measurements are $\{w(0,t),w_t(0,t),w(1,t)\}$, their result is
that the closed-loop system is asymptotically stable. The output
feedback of \cite{FG2016} uses one more measurement than
\cite{GWGBZ2013tac}. Apart from the more general external disturbance,
another point that is different here from \cite{GWGBZ2013tac,FG2016}
is that the closed-loop system in this paper is exponentially stable
and we do not require to measure the velocity $w_t(0,t)$ (or
$w_t(1,t)$) which is hard to measure \cite{FansonVmeas}. In this
paper, we only use two scalar signals (the components of $y_m$) and
this is a minimal set of measurement signals. As shown in Figure
\ref{un-wave-pic}, we apply the control force $u$ to deal with both
the internal uncertainty $f$ and the unknown external disturbance $d$.

Many control methods have been applied to deal with uncertainties in
PDE systems. The internal model principle, a classical method to cope
with uncertainty, has been generalized to infinite-dimensional systems
\cite{CIB2000tac,ReWe:2003,PauPoh:14,Nata_etal:14}. In
\cite{ReWe:2003}, the tracking and disturbance rejection problems for
infinite-dimensional linear systems, with reference and disturbance
signals that are finite superpositions of sinusoids, are
considered. The results are applied to some PDEs including the noise
reduction in a structural acoustics model described by a
two-dimensional PDE. An interesting PDE example in \cite{ReWe:2003} is
disturbance rejection in a coupled beam where the disturbance and
control are not matched. Very recently, the backstepping approach has
been used to achieve output regulation for the one-dimensional heat
equation in \cite{Deu2015auto,Deu2015tac}, and the one-dimensional
Schr\"{o}dinger equation in \cite{ZhWe:2016}. For a stochastic PDE, an
optimal control problem constrained by uncertainties in system and
control is addressed in \cite{RW-sto}. An adaptive design is exploited
in \cite{BP-Krstic2014,Krst2010MA} for dealing with the anti-stable
wave equation with unknown anti-damping coefficient. In \cite{GeHeW},
a boundary control based on the Lyapunov method is designed for the
one-dimensional Euler-Bernoulli beam equation with spatial and
boundary disturbances. However, there are not so many works, to the
best of our knowledge, on exponential stabilization (instead of
reference tracking) of PDEs with disturbance by using output
feedback. Sliding mode control that is inherently robust is the most
popular approach that can achieve exponential stability for
infinite-dimensional systems but most often, the literature considers
state feedback controllers \cite{Orlov2011siam, MCheng2011auto,
JFF2013tac,WLRCauto}, while here we aim for output feedback.

Output feedback stabilization for one-dimensional anti-stable wave
equation has been considered in \cite{JFF2015tac}, where a new type of
observer has been constructed by using three output signals to
estimate the state first and then estimate the disturbance via the
state of the observer through an {\em extended state observer} (ESO).
However, the initial state is required to be smooth in
\cite{JFF2015tac} and they obtain asymptotic stability (not
exponential, like here). In the recent paper \cite{FHGBZ} the authors
continue to investigate this question and introduce a new disturbance
estimator which is different from the traditional one, the smoothness
requirement on the initial state being removed. In \cite{FHGBZ}, still
three output signals are used as inputs to the controller and the
controller achieves asymptotic stability of the closed-loop system. In
this paper we consider the output feedback stabilization for a
one-dimensional unstable (or anti-stable) wave equation by using two
signals only, which is an improvement, and in addition we achieve
exponential stability of the state of the controlled original systems,
which is stronger than asymptotic stability.

%
%
%

Define the operators $\aline:\Dscr(\aline)\rarrow\hline$, \
$\bline_1,\bline_2:\cline\rarrow\Dscr(\aline^*)'$ by
$$ \left\{\begin{array}{l} \aline(\phi,\psi) \m=\m (\psi,\phi'')
   \FORALL (\phi,\psi)\in\Dscr(\aline),\crr\disp \Dscr(\aline) =
   \Big\{(\phi,\psi)\in H^2(0,1)\times H^1(0,1)\ |\
   \phi'(0)=\phi(0),\;\phi'(1)=0\Big\},\crr\disp \bline_1 = (0,
   -\delta_0),\ \ \ \ \bline_2=(0,\delta_1), \end{array}\right.$$
where $\delta_a$ is the Dirac pulse at $x=a$, with a suitable
interpretation. It can be shown (see \cite[Example 5.2]
{Nata_etal:16} for details) that $\Dscr(\aline^*)=\Dscr(\aline)$,
$\aline^*=-\aline$ and
\begin{equation}
   \bline_1^*(\phi,\psi) \m=\m -\psi(0),\ \ \ \bline_2^*(\phi,\psi)
   \m=\m \psi(1) \FORALL (\phi,\psi)\in\Dscr(\aline^*).
\end{equation}
We often write a pair $(a,b)$ as a column vector $\sbm{a\\ b}$. The
system \dref{wave-o} can be rewritten as
\begin{equation} \label{Trump}
   \dfrac{\dd}{\dd t}\bbm{w(\cdot,t)\\ w_t(\cdot,t)} =\m \aline\bbm
   {w(\cdot,t)\\ w_t(\cdot,t)} - \bline_1((q+1)w(0,t)) + \bline_2
   \left[ f(\sbm{w(\cdot,t)\\ w_t(\cdot,t)})+u(t)+d(t) \right].
\end{equation}
The equivalence is meant in the algebraic sense, without any reference
to existence or uniqueness of solutions, see Remark 10.1.4 in
\cite{obs_book}. The proof of the equivalence between \dref{wave-o}
and \dref{Trump} uses the theory of boundary control systems in
\cite[Section 10.1] {obs_book}, and the details (for a slightly
different system) are in \cite[Example 5.2]{Nata_etal:16}, where the
notation $B_\Nscr$ and $B$ is used in place of $\bline_1$ and
$\bline_2$ (in this order). About existence and uniqueness of
solutions we have the following proposition, whose proof is given in
the Appendix.

\begin{proposition} \label{Pro01}
The above operator $\aline$ generates a unitary group on $\hline$ and
$\bline_1,\m\bline_2$ are admissible control operators for it. Suppose
that $f:\hline\to\rline$ satisfies a global Lipschitz condition on
$\hline$ and $f(0,0)=0$. Then for any $(w_0,w_1)\in\hline$ and $u,d
\in L^2_{loc}[0,\infty)$, there exists a unique global solution to
{\rm\dref{wave-o}} such that
$(w(\cdot,t),w_t(\cdot,t))\in C(0,\infty;\hline)$.
\end{proposition}

The paper is organized as follows: We consider the exponential
stabilization of the unstable wave equation \dref{wave-o} in Sections
\ref{Sec-dis-design} to \ref{Sec-closed-U-N}. More precisely, in
Section \ref{Sec-dis-design} we desgin an infinite-dimensional total
disturbance estimator that does not use high gain, for the system
\dref{wave-o}. We propose a state observer based on this estimator
and develop an output feedback stabilizing controller by compensating
the total disturbance in Section \ref{Sec-stateobserver-U-N}. The
exponential stability of the resulting closed-loop system for
\dref{wave-o} is proved in Section \ref{Sec-closed-U-N}. Section
\ref{Sec-US-N} is devoted to the output feedback exponential
stabilization of the alternative anti-stable wave equation mentioned
earlier (with the negative damper).

\section{Disturbance estimator design} \label{Sec-dis-design}

In this section, our objective is to design a total disturbance
estimator using the input and output signals of the system
\dref{wave-o}.

\begin{remark} \label{exam-ode} {\rm
We explain the need for a disturbance estimator on a simple
finite dimensional example. Let $A\in\rline^{n\times n}$, $B\in
\rline^n$. Consider the system \vspace{-1mm}
\begin{equation}
   \dot{x}(t) \m=\m Ax(t)+Bd(t)
\end{equation}
where $x(t)\in\rline^n$ is the state trajectory at time $t$ and
$d(t)\in\rline$ is the disturbance signal at time $t$. Suppose that
$A$ is stable (Hurwitz). The solution is given by
$$x(t) - e^{At}x(0) \m=\m \int_0^t e^{A(t-s)}d(s) \dd s
   \m=\m e^{A\frac{t}{2}} \int_0^{\frac{t}{2}} e^{A(\frac{t}{2}-s)}
   d(s) \dd s + \int_{\frac{t}{2}}^t e^{A(t-s)} d(s) \dd s \m.$$
From here, it is easy to verify that $x(t)\to 0$ as $t\to\infty$ if
$d\in L^2[0,\infty)$. Therefore, to design a stabilizing control law
for $\dot{x}(t)=Ax(t)+B[u(t)+d(t)]$, it suffices to find a control
law that generates $u$ such that $u+d\in L^2[0,\infty)$.}
\end{remark}

For many boundary control systems, the control operator $B$ is
unbounded but admissible for the underlying operator semigroup. For
more on the admissibility concept we refer for instance to
\cite{obs_book}. When $x$ takes values in a Hilbert space $X$, $A$
generates an exponentially stable operator semigroup on $X$ and $B$
is admissible, we still have a stability result similar to Remark
\ref{exam-ode}, see the following lemma. For related results see
\cite{LogemannTownley97siam,Jacob2016}. As is customary, we denote by
$X_{-1}$ the dual of $\Dscr(A^*)$ with respect to the pivot space $X$,
see \cite{obs_book}.

\begin{lemma} \label{Lem-ABu}
Let $A$ be the generator of an exponentially stable operator semigroup
$e^{At}$ on the Hilbert space $X$. Assume that $B_i\in\Lscr(U_i,
X_{-1})$, $i=1,2,\ldots\m n$ are admissible control operators for
$e^{At}$ ($U_i$ are Hilbert spaces). Then the initial value problem
\vspace{-2mm}
$$ \dot{x}(t) \m=\m Ax(t)+\sum_{i=1}^nB_iu_i(t),\ \ \ x(0)=x_0,\ \ \
   u_i\in L^2_{loc}([0,\infty),U_i),\vspace{-1mm}$$
admits  a unique solution $x\in C(0,\infty;X)$, and if $u_i\in
L^\infty([0,\infty),U_i)$, $i=1,2,\ldots\m n$, then $x$ is bounded.
If for each index $i$, either $u_i\in L^2([0,\infty),U_i)$ or
$\lim_{t\to\infty}\|u(t)\|_{U_i}=0$ holds, then $x(t)\rarrow 0$ as
$t\to\infty$. Moreover, if there exist two constants $M_0,\mu_0>0$
such that $\|u\|_{U_i}\leq M_0e^{-\mu_0 t}$, $i=1,2,\ldots\m n$,
then \m $\|x(t)\|_X\leq Me^{-\mu t}$ for some $M,\mu>0$.
\end{lemma}

\begin{proof} Due to the admissibility, by \cite[Proposition 4.2.5.]
{obs_book}, the solution $x$ is a continuous $X$-valued function of
$t$ given by \vspace{-3mm}
$$ \m\ \ x(t) \m=\m e^{At}x_0+\sum_{i=1}^n\int_0^t e^{A(t-s)} B_i
   u_i(s) \dd s \m.\vspace{-1mm}$$
By assumption, there exist constants $M_1,\mu_1>0$ such that $\|
e^{At}\|\leq M_1e^{-\mu_1 t}$ for all $t\geq 0$. Thus, by
superposition, we only have to prove the statements in the lemma for
one of the integral terms in the above sum, $x_i(t)=\int_0^t e^{A
(t-s)}B_i u_i(s)\dd s$ (with $i$ fixed).

Suppose that $u_i\in L^\infty([0,\infty),U_i)$. Since ${B_i}$ is
$L^\infty$-admissible for $e^{At}$ by virtue of \cite[Remark 4.7]
{art01}, it follows from \cite[Remark 2.6]{art01} that there exists a
constant $L_1>0$ independent of $u_i$ and of $t$ such that $x_i$ is
bounded: \m $\|x_i(t)\|_X\leq L_1\|u_i\|_{L^\infty([0,\infty),U_i)}$.

Now suppose that $u_i\in L^2([0,\infty),U_i)$ or $\lim_{t\to\infty}\|u_i
(t)\|_{U_i}=0$. For any $\sigma>0$, there exists $t_\sigma>0$ such
that \vspace{-2mm}
$$ \m\ \ \ \|u_i\|_{L^2([t_\sigma,\infty),U_i)} \m\leq\m \sigma,\
   \mbox{ or }\ \ \ \|u_i\|_{L^\infty([t_\sigma,\infty),U_i)} \m\leq\m
   \sigma \m.$$
If $u_i\in L^2([0,\infty),U_i)$ then it follows from \cite[Remark 2.6]
{art01} that for any $t\geq t_\sigma$, \vspace{-2mm}
\begin{equation} \label{wAd-varep}
   \left\| \int_{t_\sigma}^t e^{A(t-s)}B_i u_i(s)\dd s\right\|_{X}
   \m\leq\m L_2\|u_i\|_{L^2([t_\sigma,\infty),U_i)} \m\leq\m L_2
   \sigma,\vspace{-1mm}
\end{equation}
where $L_2$ is a constant that is independent of $u_i$ and of $t$. If
$\lim_{t\to\infty}\|u_i(t)\|_{U_i}=0$, then by \cite[Remark 2.6]
{art01}, the $L^\infty$-admissibility of $B_i$ implies that
for any $t\geq t_\sigma$, \vspace{-2mm}
\begin{equation} \label{wAd-varep-u0}
   \left\|\int_{t_\sigma}^t e^{{A}(t-s)}B_i u_i(s)\dd s\right\|_X
   \m\leq\m L_1\|u_i\|_{L^\infty([t_\sigma,\infty),U_i)} \m\leq\m L_1
   \sigma \m.\vspace{-1mm}
\end{equation}
Using the exponential stability of $e^{At}$ again, we have that for
any $t\geq t_\sigma$, \vspace{-1mm}
\begin{equation} \label{wAd-varep-ea}
   \|e^{A(t-t_\sigma)} x_i(t_\sigma)\|_X \m\leq\m M_1 e^{-\mu_1
   (t-t_\sigma)} \| x_i(t_\sigma) \|_X \m.
\end{equation}
Since \m $x_i(t)=e^{A(t-t_\sigma)}x_i(t_\sigma)+\int_{t_\sigma}^t
e^{A(t-s)}B_i u_i(s)\dd s$, \m it follows from \dref{wAd-varep} or
\dref{wAd-varep-u0}, and \dref{wAd-varep-ea} that for $t\geq
t_\sigma$, \vspace{-4mm}
$$ \m\ \ \ \|x_i(t)\|_X \m\leq\m M_1 e^{-\mu_1 (t-t_\sigma)} \| x_i
   (t_\sigma)\|_X + \max\{L_1,L_2\}\sigma \m.$$
This shows that $\limsup_{t\to\infty}\|x(t)\|_X\leq\max\{L_1,L_2\}
\sigma$\m. Since $\sigma>0$ was arbitrary, we conclude that the last
limsup is 0, whence $x(t)\to 0$ as $t\to\infty$.

For the last part of the lemma, suppose that there exist $M_0,\mu_0>0$
such that $\|u_i\|_{U_i}\leq M_0e^{-\mu_0t}$. Choose a number
$\mu\in(0,\min\{\mu_0,\mu_1\})$, then $A+\mu I$ still generates an
exponentially stable operator semigroup. Define the functions
$x_i^\mu$ and $u_i^\mu$ by \vspace{-1mm}
$$ x_i^\mu(t) \m=\m e^{\mu t}x_i(t) \m,\qquad u_i^\mu(t) \m=\m
   e^{\mu t}u_i(t) \m,$$
then it is easy to see that the differential equation $\dot x_i^\mu=
(A+\mu I)x_i^\mu+B_i u_i^\mu$ holds. Since $u_i^\mu$ is bounded, by
an argument used at the beginning of this proof (with $x_i^\mu$ and
$u_i^\mu$ in place of $x_i$ and $u_i$), there exists $L_3>0$ such that
\m $\|x_i^\mu(t)\|_X\leq L_3\|u_i^\mu\|_{L^\infty([0,\infty),U_i)}$.
Clearly this implies that $x_i$ tends to zero at the exponential rate
$\mu$. \end{proof}

Now we design a {\em total disturbance estimator} for the system
\dref{wave-o}. This is an infinite dimensional system whose state
consists of the functions $v,v_t,z,z_t,W$ defined on $(0,1)$:
\begin{equation} \label{wave-o-v}
   \left\{\begin{array}{rl} v_{tt}(x,t) \nm &=\m v_{xx}(x,t),\crr
   \disp v_x(0,t) \nm &=\m -qw(0,t)+c_1[v(0,t)-w(0,t)],\ \ \ \
   v_x(1,t) =\m u(t)-W_x(1,t),\crr\disp v(x,0) \nm &=\m v_0(x),\
   \ \ \ v_t(x,0) \m=\m v_1(x),\crr\disp z_{tt}(x,t) \nm &=\m
   z_{xx}(x,t),\crr\disp z_x(0,t) \nm &=\m \frac{c_1}{1-c_0}z(0,t)
   +\frac{c_0}{1-c_0} z_t(0,t),\ \ \ \ z(1,t) =\m -v(1,t)+w(1,t) -
   W(1,t),\crr\disp z(x,0) \nm &=\m z_0(x),\ \ \ \ z_t(x,0) =\m
   z_1(x),\crr\disp W_t(x,t) \nm &=\m -W_x(x,t),\crr\disp W(0,t)
   \nm &=\m -c_0[v(0,t)-w(0,t)],\qquad W(x,0) \m=\m W_0(x),
   \end{array}\right.
\end{equation}
where $c_0$ and $c_1$ are two positive design parameters, $c_0<1$,
$(v_0,v_1,z_0,z_1,W_0)\in\hline^2\times H^1(0,1)$ is the initial state
of the disturbance estimator and its input signals are $u$, $w(0,t)$
and $w(1,t)$. The output of this estimator is $\widehat
F(t)=z_x(1,t)$.

\begin{remark} \label{overview_TDE} {\rm
Before going into the tedious technical details, we give an informal
overview of how the total disturbance estimator \dref{wave-o-v}
works. The ``$(v,W)$-part'' of \dref{wave-o-v} is used to channel the
total disturbance from the original system to an exponentially stable
wave equation with state $(p,p_t)$, where $p=w-v-W$, described in
\dref{wave-o-hatv}. (The equations \dref{wave-o-hatv} contain also a
$W$-part, but from an input-output point of view, this $W$-part is
irrelevant.) The effect of $u$ is cancelled in the estimator, so that
$u$ has no influence on $p$. The wave equation system with state
$(p,p_t)$ has input $F$ and output $p(1,t)$ and it represents from an
input-output view the linear part of the plant and the
``$(v,W)$-part'' of \dref{wave-o-v}, taken together, see Figure 2.
This is a well-posed boundary control system (in the sense of
\cite[Definition 10.1.7]{obs_book}), with a bounded observation
operator, so that for large $\Re s$, its transfer function $\GGG$
satisfies $|\GGG(s)|\leq m (\Re s)^{-\half}$, see for instance
\cite[Proposition 4.4.6]{obs_book}.

The $z$-part of \dref{wave-o-v} is in fact the same boundary control
system as the one just described, but with the roles of input and
output reversed. This would be flow inversion in the sense of
\cite{StWe:04}, except that the $z$-part is ill-posed. Indeed, its
transfer function is $\GGG^{-1}$, and from our estimate on $\GGG$ it
follows that $\GGG^{-1}$ is not proper. Overall, the transfer function
from $F$ to $\widehat F$ is the constant 1. The difference $\widehat F
-F$ depends linearly on the deviation between the initial state of the
$z$-part of \dref{wave-o-v} and the initial state of the $p$-part of
\dref{wave-o-hatv}. Since the $z$-part, in the absence of any input
(i.e., when $p(1,t)\equiv 0$) is exponentially stable, and its
observation operator giving $\widehat F$ is admissible (as we shall
see in Lemma \ref{dis-est}), it follows that $\hat F-F\in L^2[0,
\infty)$. The overall linear system shown in Figure 2 (with input
$(F,u)$ and output $\widehat F$) is well-posed. If $f$ is globally
Lipschitz, then also the overall nonlinear system (with input $(d,u)$
and output $\widehat F$) is well-posed (due to Proposition
\ref{Pro01}).} \end{remark}

\vspace{-5mm}
\begin{center}
\includegraphics[scale=0.34]{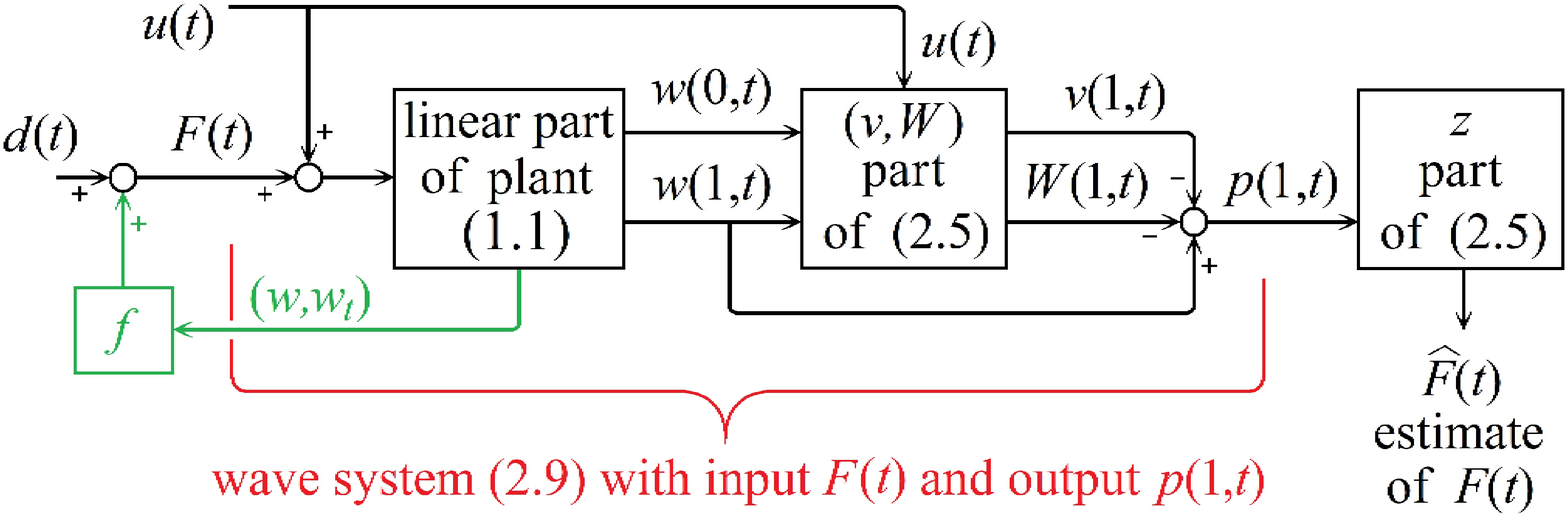}
\end{center} \vspace{-9mm}
{\leftskip 11mm {\rightskip 11mm \noindent
Figure 2. The total disturbance estimator connected to the plant.
The $z$-part of the disturbance estimator \dref{wave-o-v} is the
(ill-posed) flow inverse of the wave system \dref{wave-o-hatv} (which
has input $F$ and output $p(1,t)$). The system with input $(F,u)$ and
output $\widehat F$ is linear and its transfer function is $[1\ 0]$.
\bigskip \par}}

Now we start providing the technical details for the operation of
the total disturbance estimator. Consider the plant \dref{wave-o}
coupled with the estimator \dref{wave-o-v} and denote
\begin{equation} \label{hatv-def}
   \widehat{v}(x,t) \m=\m v(x,t)-w(x,t).
\end{equation}
Then it is easy to verify that the subsystem with state $(\widehat{v}
(x,t),W(x,t))$ satisfies
\begin{equation} \label{wave-o-hatv-P}
   \left\{\begin{array}{l} \widehat{v}_{tt}(x,t) \m=\m \widehat{v}
   _{xx}(x,t),\crr\disp \widehat{v}_x(0,t)=c_1\widehat{v}(0,t),\ \ \
   \widehat{v}_x(1,t)+W_x(1,t) \m=\m -F(t),\crr\disp W_t(x,t) \m=\m
   -W_x(x,t),\ \ \ W(0,t) \m=\m -c_0\widehat{v}(0,t) \m,\end{array}
   \right.
\end{equation}
where $F$ is the total disturbance from \dref{Jared}. It will be
convenient to change variables once more, by introducing the
notation \vspace{-2mm}
\begin{equation} \label{Shavuot}
   p(x,t) \m=\m -\widehat v(x,t) - W(x,t) \m,\qquad \widetilde c_0
   \m=\m \frac{c_0}{1-c_0} \m,\qquad \widetilde c_1 \m=\m
   \frac{c_1}{1-c_0} \m,
\end{equation}
then from the last part of \dref{wave-o-hatv-P} we see that $p(0,t)
=-(1-c_0)\widehat v(0,t)$ and hence (using that $-W_x(0,t)=W_t(0,t)$)
the subsystem with state $(p(\cdot,t),W(\cdot,t))$ is governed by
\begin{equation} \label{wave-o-hatv}
   \left\{\begin{array}{l} p_{tt}(x,t) \m=\m p_{xx}(x,t),\crr\disp
   p_x(0,t) \m=\m \widetilde c_1 p(0,t) + \widetilde c_0 p_t(0,t),\
   \ \ p_x(1,t) \m=\m F(t),\crr\disp W_t(x,t) \m=\m -W_x(x,t),\ \ \
   \ W(0,t) \m=\m \widetilde c_0 p(0,t),\end{array}\right.
\end{equation}
with the initial state $p(x,0)=-\widehat{v}(x,0)-W(x,0)$,
$p_t(x,0)=-\widehat{v}_t(x,0)+W_x(x,0)$. The following lemma states
some stability properties of the system \dref{wave-o-hatv-P}.

\begin{lemma} \label{Lem-hatz-bd}
Suppose that $d\in L^\infty[0,\infty)$ (or $d\in L^2[0,\infty)$),
$f:\hline\to\rline$ is continuous and that {\rm\dref{wave-o}} admits a
unique solution $(w,w_t)\in C(0, \infty;\hline)$ which is bounded.
For any initial state $(\widehat{v}_0,\widehat{v}_1,W_0)\in\hline
\times H^1(0,1)$ with the compatibility condition $W_0(0)=-c_0\widehat
{v}_0(0)$, there exists a unique solution $(\widehat{v},\widehat{v}_t,
W)\in C(0,\infty;\hline \times H^1(0,1))$ to {\rm\dref{wave-o-hatv-P}}
and
\begin{equation} \label{lem-v-W-bd}
   \sup_{t\geq 0}\|(\widehat{v}(\cdot,t),\widehat{v}_t(\cdot,t),W
   (\cdot,t))\|_{\hline\times H^1(0,1)} \m<\m \infty \m.
\end{equation}
If we assume further that $\lim_{t\to\infty}|f(w,w_t)|=0$ and $d\in
L^2[0,\infty)$, then
\begin{equation} \label{lem-v-W-bd0}
   \lim_{t\to\infty} \|(\widehat{v}(\cdot,t),\widehat{v}_t(\cdot,t),
   W(\cdot,t))\|_{\hline\times H^1(0,1)} \m=\m 0 \m.
\end{equation}
If we assume  that $f\equiv 0$ and $d\equiv 0$, then there exist two
constants $M,\mu>0$ such that
\begin{equation} \label{lem-v-W-bd-decay}
   \|(\widehat{v}(\cdot,t),\widehat{v}_t(\cdot,t),W(\cdot,t))\|
   _{\hline\times H^1(0,1)} \m\leq\m Me^{-\mu t}\ \ \ \forall \m
   t\geq 0 \m.
\end{equation}
\end{lemma}

\begin{proof} We shall use the equivalent system \dref{wave-o-hatv}.
We define the operators $\Ascr$ and $\Bscr$ (that resemble $\aline$
and $\bline_2$ from \dref{Trump}) by
\begin{equation} \label{bbAep-def}
   \left\{\nm\begin{array}{l} \Ascr(\phi,\psi) \m=\m (\psi,\phi'')
   \FORALL (\phi,\psi)\in\Dscr(\Ascr) \m,\qquad \Bscr \m=\m
   (0,\delta_1),\crr\disp \Dscr(\Ascr) = \Big\{(\phi,\psi) \in
   H^2(0,1)\times H^1(0,1)\m|\ \phi'(0)=\widetilde c_1\phi(0) +
   \widetilde c_0 \psi(0) \m,\ \phi'(1)=0\Big\}.\end{array}\right.
\end{equation}
Then the ``$p$-part'' of \dref{wave-o-hatv} can be
written in abstract form as
$$ \dfrac{\dd}{\dd t}\bbm{p(\cdot,t)\\ p_t(\cdot,t)} \m=\m \Ascr
   \bbm{p(\cdot,t)\\ p_t(\cdot,t)} + \Bscr \left[ f(\sbm{w(\cdot,t)\\
   w_t(\cdot,t)}) + d(t) \right].$$
It is well-known \cite[Theorem 2.1]{GuoXu2007} that $\Ascr$ generates
an exponentially stable operator semigroup $e^{\Ascr t}$ on $\hline$
and $\Bscr$ is admissible for $e^{\Ascr t}$. Since $f:\hline\to\rline$
is continuous and $(w,\dot{w})\in C(0,\infty;\hline)$ is bounded, we
have $f(w,\dot w)\in L^\infty[0,\infty)$. Thus, by $d\in L^\infty[0,
\infty)$ or by $d\in L^2[0,\infty)$, it follows from Lemma
\ref{Lem-ABu} that the ``$p$-part'' of \dref{wave-o-hatv} admits a
unique bounded solution, so that there exists a constant $M_1>0$ such
that \vspace{-1mm}
\begin{equation} \label{til-v-bd}
   \sup_{t\geq 0} \|(p(\cdot,t),p_t(\cdot,t))\|_\hline \leq M_1 \m.
\end{equation}

We claim that $\|W(\cdot,t))\|_{H^1(0,1)}$ is uniformly bounded for
all $t\geq 0$. To prove this, first we show that for all $t\geq 1$,
\vspace{-1mm}
\begin{equation} \label{til-v-bdt-bd}
   \int_0^1 p^2_t(0,t-x) \dd x \m\leq\m 3\max_{s\in[t-1,t]}\|
   (p(\cdot,t),p_t(\cdot,t))\|^2_\hline\m.
\end{equation}
Indeed, define \vspace{-2mm}
$$ \rho(t) \m=\m 2\int_0^1 (x-1)p_t(x,t)p_x(x,t) \dd x \m.$$
Then $|\rho(t)|\leq 2\|p_t(\cdot,t)\|_{L^2}\|p_x(\cdot,t)\|_{L^2}
\leq\|(p(\cdot,t),p_t(\cdot,t))\|_\hline^2$. Computing $\frac{\dd}
{\dd t}\rho(t)$ along the solution of the ``$p$-part'' of
\dref{wave-o-hatv}, using that $2\frac{\dd}{\dd t}[p_t p_x]=\frac
{\dd}{\dd x}\left[p_x^2+p_t^2\right]$, yields
$$ \dot{\rho}(t)=p_x^2(0,t)+p_t^2(0,t)-\int_0^1 [p_x^2(x,t)+p_t^2
   (x,t)] \dd x \geq p_t^2(0,t) - \int_0^1 [p_x^2(x,t)+p_t^2(x,t)]
   \dd x,$$
which  implies that, for $t\geq 1$,
\begin{equation} \label{wtv-3norm}
   \int_{t-1}^t p_s^2(0,s) \dd s \m\leq \int_{t-1}^t \|(p(\cdot,
   s),p_s(\cdot,s))\|_\hline^2 \dd x + \rho(t)-\rho(t-1) \leq\m
    3\max_{s\in[t-1,t]}\|(p(\cdot,t),p_t(\cdot,t))\|^2_\hline\m.
\end{equation}
On the other hand, since for any $t\geq 1$, \ $\int_0^1 p^2_t(0,t-x)
\dd x=\int_{t-1}^t p_s^2(0,s)\dd s$, we obtain \dref{til-v-bdt-bd}.
Define the function
\vspace{-2mm}
\begin{equation} \label{W-sol}
   W(x,t) \m=\m \left\{\begin{array}{ll}\disp \widetilde c_0 p(0,
   t-x),& t\geq x,\crr\disp W_0(x-t),& x>t.\end{array}\right.
\vspace{-1mm} \end{equation}
Then a simple computation shows that $W$ solves the ``$W$-part'' of
\dref{wave-o-hatv}. It follows from the Sobolev embedding theorem,
the last part of \dref{wave-o-hatv} and \dref{til-v-bd} that
\vspace{-1mm}
\begin{equation} \label{W-bd0-est}
   |W(0,t)| =\m \widetilde c_0|p(0,t)| \leq \widetilde c_0\|
   p(\cdot,t)\|_{H^1(0,1)} \leq\m \widetilde c_0\|(p(\cdot,t),p_t
   (\cdot,t))\|_\hline \leq \widetilde c_0 M_1 \m.
\end{equation}
From \dref{W-sol} we derive that for $t\geq 1$, \m $\int_0^1 W_x^2
(x,t)\dd x=\widetilde c_0^2\int_0^1 p^2_t(0,t-x)\dd x$. Then the
boundedness of $\|W(\cdot,t)\|_{H^1(0,1)}$ follows from here, using
\dref{til-v-bd}, \dref{W-bd0-est} and \dref{til-v-bdt-bd}.

Since $\widehat{v}(x,t)=-p(x,t)-W(x,t)$ and $W_t(x,t)=-W_x(x,t)$,
we have
$$ \sup_{t\geq 0}\|(\widehat{v}(\cdot,t),\widehat{v}_t(\cdot,t))\|
   _\hline \leq \sup_{t\geq 0} \big[ \|(p(\cdot,t),p_t(\cdot,t))\|
   _\hline + \|(W(\cdot,t),W_x(\cdot,t))\|_\hline \big] \m.$$
This with \dref{til-v-bd} and the boundedness of $\|W(\cdot,t)\|
_{H^1(0,1)}$ implies that \dref{lem-v-W-bd} holds.

Next, suppose that $\lim_{t\to\infty}|f(w,w_t)|=0$ and $d\in L^2[0,
\infty)$. It follows from Lemma \ref{Lem-ABu} that the ``$p$-part''
of \dref{wave-o-hatv} admits a unique solution satisfying
\begin{equation} \label{til-v-bd0}
   \lim_{t\to\infty} \|(p(\cdot,t),p_t(\cdot,t))\|_\hline =\m 0 \m.
   \vspace{-2mm}
\end{equation}
By \dref{til-v-bdt-bd} and \dref{til-v-bd0}, we get $\int_0^1 p_t^2
(0,t-x)\dd x\to 0$ as $t\to\infty$. Then from \vspace{-2mm}
\begin{equation} \label{W-x-est}
   \|W(\cdot,t)\|_{H^1(0,1)}^2 =\m |W(0,t)|^2 + \int_0^1 W_x^2(x,t)
   \dd x \m=\m \widetilde c_0^2 \left[ p^2(0,t) + \int_0^1 p_t^2(0,
   t-x) \dd x \right]
\end{equation}
we see that $\lim_{t\rarrow\infty}\|W(\cdot,t)\|_{H^1(0,1)}=0$.
This, with $\widehat{v}(x,t)=-p(x,t)-W(x,t)$, $W_t(x,t)=-W_x(x,t)$
and \dref{til-v-bd0}, gives \dref{lem-v-W-bd0}.

Next, suppose that $f\equiv 0$ and $d\equiv 0$. Since $\Ascr$
generates an exponentially stable operator semigroup $e^{\Ascr t}$
on $\hline$, there exist two constants $M_3,\mu_3>0$ such that
\begin{equation} \label{p-expdecay}
   \|(p(\cdot,t),p_t(\cdot,t))\|_\hline \m\leq\m M_3e^{-\mu_3t}
   \FORALL t\geq 0 \m.
\end{equation}
Since by \dref{til-v-bdt-bd} and \dref{p-expdecay} we have $\int_0^1
p_t^2(0,t-x)\dd x\leq 3M_3^2e^{-2\mu_3t}$, it follows from
\dref{W-bd0-est}, \dref{W-x-est} and \dref{p-expdecay} that $\|W
(\cdot,t)\|_{H^1(0,1)}$ also converges to zero exponentially.
Combining this with $\widehat{v}(x,t)=-p(x,t)-W(x,t)$ and $W_t(x,t)=
-W_x(x,t)$, we get \dref{lem-v-W-bd-decay}. \end{proof}

To understand that the ``$z$-part'' of \dref{wave-o-v} is used to
invert the system \dref{wave-o-hatv}, denote \vspace{-1mm}
\begin{equation} \label{tildez-def}
   \beta(x,t) \m=\m z(x,t)-p(x,t) \m.
\end{equation}
Still using the notation \dref{Shavuot}, we can see that $\beta(x,
t)$ is governed by
\begin{equation} \label{wave-o-tildz}
   \left\{\begin{array}{l} \beta_{tt}(x,t) \m=\m \beta_{xx}(x,t),
   \crr\disp \beta_x(0,t) \m=\m \widetilde c_1\beta(0,t) +
   \widetilde c_0\beta_t(0,t),\ \ \ \beta(1,t) \m=\m 0 \m.
   \end{array}\right.
\end{equation}
We consider the system \dref{wave-o-tildz} in the energy Hilbert state
space $\hline_0=H^1_R(0,1)\times L^2(0,1)$, where $H^1_R(0,1)=\{\phi
\in H^1(0,1):\phi(1)=0\}$, with the usual inner product from
\dref{Eliel}, so that $\hline_0$ is a closed subspace of $\hline$.
The system \dref{wave-o-tildz} can be rewritten as \vspace{-2mm}
$$ \frac{\dd}{\dd t}(\beta(\cdot,t),\beta_t(\cdot,t)) \m=\m \Ascr_0
   (\beta(\cdot,t),\beta_t(\cdot,t)) \m,$$
where \vspace{-2mm}
\begin{equation} \label{bbA0ep-def}
   \Ascr_0(\phi,\psi) \m=\m (\psi,\phi'') \FORALL (\phi,\psi) \in
   \Dscr(\Ascr_0) \m,\ \ \ \ \ \Dscr(\Ascr_0) \m=\vspace{-2mm}
\end{equation}
$$ \bigg\{(\phi,\psi)\in H^2(0,1)\times
   H^1(0,1)\ \bigg|\ \phi(1)=0,\ \psi(1)=0,\ \phi'(0) = \widetilde
   c_1 \phi(0) + \widetilde c_0\psi(0)\bigg\}.$$

It is well-known (\cite[Theorem 3]{ChengG1979}) that $\Ascr_0$
generates an exponentially stable operator semigroup $e^{\Ascr_0 t}$
on $\hline_0$. Thus, for any initial state $(\beta_0,\beta_1)\in
\hline_0$, \dref{wave-o-tildz} has a unique solution $(\beta(\cdot,
t),\beta_t(\cdot,t))=e^{\Ascr_0 t}(\beta_0,\beta_1)\in C(0,\infty;
\hline_0)$, and this decays exponentially.

\begin{lemma} \label{dis-est}
The observation operator $C:\Dscr(\Ascr_0)\rarrow\rline$ defined by
$C(\beta_0,\beta_1)=(\frac{\dd}{\dd x}\beta_0)(1)$ is admissible for
the operator semigroup $e^{\Ascr_0 t}$.
\end{lemma}

\begin{proof} Consider the semigroup generator $\Ascr_1$ on $X=H^1
_0(0,1)\times L^2(0,1)$ by the same formula as $\Ascr_0$, but with
domain $\Dscr(\Ascr_1)=[H^2(0,1)\cap H^1_0(0,1)]\times H^1_0(0,1)$. It
is well-known that $C$ is admissible for $e^{\Ascr_1 t}$, see for
instance \cite[Proposition 6.2.1]{obs_book}. Take $(\beta_0,\beta_1)
\in\Dscr(\Ascr_0)\cap\Dscr(\Ascr_1)$, which is dense in $\hline_0$.
By the result just mentioned, the function $y:[0,\half]\rarrow\rline$
defined by $y(t)=Ce^{\Ascr_1 t}(\beta_0,\beta_1)$ is in $L^2[0,
\half]$ and there is a $k\geq 0$ (independent of $(\beta_0,\beta_1)$)
such that $\|y\|_{L^2}\leq k\|(\beta_0,\beta_1)\|_X$. Notice that
$\|(\beta_0,\beta_1)\|_X=\|(\beta_0,\beta_1)\|_{\hline_0}$. Because
information in solutions of the wave equation propagates with speed at
most 1, the left boundary condition has no influence on $y$, so that
we have $y(t)=Ce^{\Ascr_0 t}(\beta_0,\beta_1)$. This fact, together
with our estimate on $\|y\|_{L^2}$, proves that $C$ is admissible also
for $\Ascr_0$. \end{proof}

\begin{remark} \label{Rem-k-est-d} {\rm
Since $C$ is admissible for $e^{\Ascr_0 t}$ and this operator
semigroup is exponentially stable, it follows (see \cite[Remark 4.3.5]
{obs_book}) that the function $y(t)=Ce^{\Ascr_0 t}(\beta_0,\beta_1)$
is in $L^2[0,\infty)$, for any $(\beta_0,\beta_1)\in\hline_0$. In
terms of solutions of \dref{wave-o-tildz}, $y(t)=\beta_x(1,t)$. From
\dref{tildez-def} $\beta_x(1,t)=z_x(1,t)-p_x(1,t)$. Now using the
third equation in \dref{wave-o-hatv}, we get $\beta_x(1,t)=\widehat
F(t)-F(t)$. Thus, $\widehat F$ can be regarded as an estimate of $F$,
because $\widehat F-F\in L^2[0,\infty)$.}
\end{remark}

\section{Controller and observer design}\label{Sec-stateobserver-U-N}

In this section, based on our disturbance estimator, we design a state
observer for the system \dref{wave-o} as follows:
\begin{equation} \label{wave-o-obser}
   \left\{\begin{array}{rl} \widehat{w}_{tt}(x,t) \nm &=\m \widehat
   {w}_{xx}(x,t),\crr\disp \widehat{w}_x(0,t) \nm &=\m -qw(0,t)+c_1
   [\widehat{w}(0,t)-w(0,t)],\crr\disp \widehat{w}_x(1,t) \nm &=\m
   u(t)+\widehat F(t)-Y_x(1,t),\crr\disp \widehat{w}(x,0) \nm &=\m
   \widehat{w}_0(x),\ \ \widehat{w}_t(x,0)\m=\m \widehat{w}_1(x),\crr
   \disp Y_t(x,t) \nm &=\m -Y_x(x,t),\crr\disp Y(0,t) \nm &=\m -c_0
   [\widehat{w}(0,t)-w(0,t)],\qquad Y(x,0) \m=\m Y_0(x),
   \end{array}\right.
\end{equation}
where $c_0$ and $c_1$ are the same as in \dref{wave-o-v} and $\widehat
F(t)=z_x(1,t)$ is generated by the total disturbance estimator
\dref{wave-o-v}. The system \dref{wave-o-obser} is a ``natural
observer'' \cite{De2004scl} after canceling the disturbance, in the
sense that it employs a copy of the plant plus output injection (in
this case, only at the boundary). Note that the observer
\dref{wave-o-obser} is different from the one in \cite{KGBS:08}, where
the signal $w_t(1,t)$ (that is unavailable here) is used.

To show the asymptotic convergence of the above observer, we introduce
the observer error variable \vspace{-2mm}
\begin{equation} \label{varep-def}
   \e(x,t) \m=\m \widehat{w}(x,t)-w(x,t) \m.
\end{equation}
Then, using the notation $\beta$ from \dref{tildez-def}, $(\e(x,t),
Y(x,t))$ satisfies
\begin{equation} \label{Boris-P}
   \left\{\begin{array}{l} \e_{tt}(x,t) \m=\m \e_{xx}(x,t),\crr\disp
   \e_x(0,t) \m=\m c_1\e(0,t), \ \ \ \e_x(1,t) \m=\m \beta_x(1,t)-
   Y_x(1,t), \crr\disp Y_t(x,t) \m=\m -Y_x(x,t),\ \ \
   Y(0,t) \m=\m -c_0\e(0,t).\end{array}\right.
\end{equation}
We have the following lemma to show that \dref{Boris-P} is
asymptotically stable.

\begin{lemma} \label{lem-varepto0}
For any initial state $(\e(\cdot,0),\e_t(\cdot,0),Y(\cdot,0))\in
\hline\times H^1(0,1)$ with the compatibility condition $Y(0,0)=-c_0
\e(0,0)$, there exists a unique solution to {\rm\dref{Boris-P}} such
that $(\e,\e_t,Y)\in C(0, \infty;\hline\times H^1(0,1))$ and it
satisfies
\begin{equation} \label{varepsi-to0}
   \lim_{t\to\infty}\|(\e(\cdot,t),\e_t(\cdot,t),Y(\cdot,t))
   \|_{\hline\times H^1(0,1)} \m=\m 0 \m.
\end{equation}
\end{lemma}

\begin{proof} We introduce a new variable \vspace{-2mm}
\begin{equation} \label{tvarep-def}
   \widetilde\e(x,t) \m=\m \e(x,t)+Y(x,t) \m.
\end{equation}
Then it is easy to check that $(\widetilde\e(x,t),Y(x,t))$ is
governed by
\begin{equation} \label{Boris}
   \left\{\begin{array}{l} \widetilde\e_{tt}(x,t) \m=\m \widetilde
   \e_{xx}(x,t),\crr\disp \widetilde\e_x(0,t) \m=\m \widetilde c_1
   \widetilde\e(0,t)+\widetilde c_0\widetilde\e_t(0,t),\ \ \
   \widetilde\e_x(1,t) \m=\m \beta_x(1,t),\crr\disp
   Y_t(x,t) \m=\m -Y_x(x,t),\ \ \ Y(0,t) \m=\m -\widetilde c_0
   \widetilde\e(0,t), \end{array}\right.
\end{equation}
with the initial state \m $\widetilde\e(x,0)=\e(x,0)+Y(x,0),\
\widetilde\e_t(x,0)=\e_t(x,0)-Y_x(x,0),\ Y(x,0)=Y(x,0)$. The
$\e$-part of the system \dref{Boris} can be rewritten as
$$ \dfrac{\dd}{\dd t} \bbm{\widetilde \e(\cdot,t)\\ \widetilde\e_t
   (\cdot,t)} \m=\m \Ascr \bbm{\widetilde\e(\cdot,t)\\ \widetilde
   \e_t(\cdot,t)} + \Bscr\beta_x(1,t),$$
where $\Ascr$ and $\Bscr$ are defined by \dref{bbAep-def}. As already
mentioned, we know from \cite{GuoXu2007} that $\Ascr$ is an
exponentially stable semigroup generator on $\hline$ and $\Bscr$ is
admissible for it. By Remark \ref{Rem-k-est-d}, $\beta_x(1,t)\in L^2
[0,\infty)$. It follows from Lemma \ref{Lem-ABu} that for any initial
state in $\hline$, \dref{Boris} has a unique solution that satisfies
\vspace{-2mm}
\begin{equation} \label{tepsi-to0}
   \lim_{t\to\infty}\|(\widetilde\e(\cdot,t),\widetilde\e_t
   (\cdot,t))\|_{\hline} \m=\m 0 \m.
\end{equation}
The remaining part of the proof is very similar to the proof of
\dref{lem-v-W-bd0}, just replace $\widehat v,W,p$ and $F$ used there
with $\e,Y,-\widetilde\e$ and $-\beta_x$ used here. \end{proof}

Lemma \ref{lem-varepto0} shows that \dref{wave-o-obser} is indeed an
observer for the system \dref{wave-o}. Now, by the observer-based
feedback control law of \cite{KGBS:08}, we propose the following
observer-based feedback controller (the motivation behind it will be
clear from \dref{hatw-tw} to \dref{tilde-w-sys-0}): \vspace{-2mm}
\begin{equation} \label{con-out}
   \begin{array}{rl}\disp u(t) &=\m -\widehat F(t)+Y_x(1,t)-c_3
   \widehat{w}_t(1,t)-(c_2+q)\widehat{w}(1,t)\crr\disp
   &\m\ \ \ \ - (c_2+q)\int_0^1e^{q(1-\xi)}[c_3\widehat{w}_t(\xi,t)+q
   \widehat{w}(\xi,t)] \dd\xi \m,\end{array}
\end{equation}
where $c_2,c_3$ are positive design parameters. The term $-\widehat
F(t)$ is used to essentially cancel the total disturbance $F(t)$ in
\dref{wave-o}, which is the estimation/cancellation strategy, and the
remaining terms are used to stabilize the system \dref{wave-o-obser}.
The closed-loop system formed of the observer \dref{wave-o-obser} and
the controller \dref{con-out} is
\begin{equation} \label{wave-o-obser-closed}
   \left\{\begin{array}{rl} \nm\widehat{w}_{tt}(x,t) \nm\nm &=\m
   \widehat{w}_{xx}(x,t),\crr\disp \nm\widehat{w}_x(0,t) \nm\nm &=\m
   -qw(0,t)+c_1 [\widehat{w}(0,t)-w(0,t)],\crr\disp \nm\widehat{w}
   _x(1,t) \nm\nm &=\m -c_3 \widehat{w}_t(1,t)-(c_2+q)\widehat{w}(1,t)
   -(c_2+q)\int_0^1 e^{q(1-\xi)}[c_3\widehat{w}_t(\xi,t)+q\widehat{w}
   (\xi,t)]\dd\xi,\crr\disp \nm Y_t(x,t) \nm\nm &=\m -Y_x(x,t),\ \ \
   Y(0,t) \m=\m -c_0 [\widehat{w}(0,t)-w(0,t)].\end{array}\right.
\end{equation}
Consider the invertible change of variable
\begin{equation} \label{hatw-tw}
   \widetilde{w}(x,t) \m=\m [(I+\mathbb{P})\widehat{w}](x,t) \m=\m
   \widehat{w}(x,t)+(c_2+q)\int_0^x e^{q(x-\xi)}\widehat{w}(\xi,t)
   \dd\xi \m,
\end{equation}
where $\mathbb{P}$ is a Volterra transformation
\cite{KGBS:08}. The inverse $(I+\mathbb{P})^{-1}$ is given by
\begin{equation}
   \widehat{w}(x,t) \m=\m [(I+\mathbb{P})^{-1}\widetilde{w}](x,t)
   \m=\m \widetilde{w}(x,t)-(c_2+q)\int_0^x e^{-c_2(x-\xi)}
   \widetilde{w}(\xi,t) \dd\xi \m.
\end{equation}
It can be shown that the transformation \dref{hatw-tw} converts
system \dref{wave-o-obser-closed} into
\begin{equation} \label{tilde-w-sys}
\left\{\begin{array}{l}
\widetilde{w}_{tt}(x,t)=\widetilde{w}_{xx}(x,t)
-(c_1+q)(c_2+q)e^{qx}\e(0,t),\crr\disp
\widetilde{w}_x(0,t)=c_2\widetilde{w}(0,t)+(c_1+q)\e(0,t),\ \
\widetilde{w}_x(1,t)=-c_3\widetilde{w}_t(1,t),  \crr\disp
Y_t(x,t)=-Y_x(x,t),\ \ Y(0,t)=-c_0\e(0,t).
\end{array}\right.
\end{equation}
Thus, the overall system is a cascade of the exponentially stable
``$(\widetilde{w},Y)$-part'' subsystem and the asymptotical stable
``$\e$-part'' subsystem. For $\e(0,t)=0$, the resulting system
\dref{tilde-w-sys} is exponentially stable:
\begin{equation} \label{tilde-w-sys-0}
   \left\{\begin{array}{l}
   \widetilde{w}_{tt}(x,t)=\widetilde{w}_{xx}(x,t),  \crr\disp
   \widetilde{w}_x(0,t)=c_2\widetilde{w}(0,t), \ \ \
   \widetilde{w}_x(1,t)=-c_3\widetilde{w}_t(1,t),  \crr\disp
   Y_t(x,t)=-Y_x(x,t),\ \ \ \ Y(0,t) =\m 0 \m.
\end{array}\right.
\end{equation}
This is a familiar form of a wave equation with a ``passive damper''
boundary condition coupled with a finite time stable transport
equation. The solution of the ``$\widetilde{w}$-part'' is
exponentially stable and the solution of the ``$Y$-part'' satisfies
$Y(x,t)\equiv0$ for $t\geq 1$. The idea of the transformation
\dref{hatw-tw} is that it makes the closed-loop system
\dref{tilde-w-sys} behave like the system \dref{tilde-w-sys-0} (in
the absence of an observer) by propagating the destabilizing
$q$-term from the boundary $x=0$, through the entire domain, to the
boundary $x=1$, where it gets cancelled by the feedback.

\begin{lemma}\label{lem-tw-to0}
Suppose that the signal $\e(0,t)$ is determined by the system {\rm
\dref{Boris-P}}. Then for any initial state $(\widetilde{w}(\cdot,0),
\widetilde{w}_t(\cdot,0),Y(\cdot,0))\in\hline\times H^1(0,1)$
satisfying the compatibility condition $Y(0,0)=-c_0\e_0(0,0)$, there
exists a unique solution to {\rm\dref{tilde-w-sys}} such that\\
$(\widetilde{w},\widetilde{w}_t)\in C(0,\infty;\hline)$ and this
solution satisfies
\begin{equation} \label{tw-si-to0}
   \lim_{t\to\infty}\|(\widetilde{w}(\cdot,t),\widetilde{w}_t(\cdot,
   t),Y(\cdot,t))\|_{\mathbb{H}\times H^1(0,1)} \m=\m 0 \m.
\end{equation}
\end{lemma}

\begin{proof} The convergence of ``$Y$-part'' of \dref{tilde-w-sys}
follows from Lemma \ref{lem-varepto0}. We can write the
``$\widetilde{w}$-part'' of system \dref{tilde-w-sys} into abstract
operator form as follows:
\begin{equation} \label{varep-abs}
   \frac{\dd}{\dd t}\bbm{\widetilde{w}(\cdot,t)\\ \widetilde{w}_t
   (\cdot,t)} \m=\m A_{\widetilde{w}} \bbm{\widetilde{w}(\cdot,t)\\
   \widetilde{w}_t(\cdot,t)} + B_1\e(0,t)+B_2\e(0,t),
\end{equation}
where the operators $A_{\widetilde{w}}:\Dscr(A_{\widetilde{w}})
(\subset\hline)\to\hline$, $B_1$ and $B_2$ are given by
\begin{equation} \label{Aw-def}
   \left\{\begin{array}{l}\disp A_{\widetilde{w}}(\phi,\psi) \m=\m
   (\psi,\phi'') \FORALL (\phi,\psi)\in\Dscr({A}_{\widetilde{w}}),
   \crr\disp \Dscr(A_{\widetilde{w}}) \m=\m \left\{(\phi,\psi)\in
   H^2(0,1)\times H^1(0,1)\ |\ \phi'(0)=c_2\phi(0),\ \phi'(1) = -c_3
   \psi(1)\right\},\crr\disp B_1=(c_1+q)(0,-\delta_0),\ \ \ \ B_2 =
   -(c_1+q)(c_2+q)(0,-e^{qx}).\end{array} \right.
\end{equation}
It is well known (\cite[Proposition 2]{JFF2013tac}) that
$A_{\widetilde{w}}$ generates an exponentially stable operator
semigroup $e^{A_{\widetilde{w}}t}$ on $\hline$ and $B_1$ is
admissible for it. On the other hand, since the operator $B_2$
is bounded, it is also admissible for $e^{A_{\widetilde{w}}t}$.
By the Sobolev embedding theorem and Lemma
\ref{lem-varepto0}, we obtain
$$ |\e(0,t)| \m\leq\m \|\e(0,t)\|_{H^1(0,1)}\to 0,\ \mbox{ as }
   t\to\infty \m.$$
It follows from  Lemma \ref{Lem-ABu} that $\lim_{t\to\infty}\|
(\widetilde{w}(\cdot,t),\widetilde{w}_t(\cdot,t))\|_{\mathbb{H}}=0$.
\end{proof}

\section{{ Well-posedness and stability of the closed-loop system}}
\label{Sec-closed-U-N} 

In this section we show the well-posedness and exponential stability
of the closed-loop system of \dref{wave-o}. First we claim that the
system \dref{tilde-w-sys} is exponentially stable. To this end, we
consider the overall system \dref{wave-o-tildz}, \dref{Boris-P} and
\dref{tilde-w-sys} as follows:
\begin{equation} \label{Lapid}
   \left\{\begin{array}{l} \e_{tt}(x,t) \m=\m \e_{xx}(x,t),\crr\disp
   \e_x(0,t)=c_1\e(0,t),\ \ \ \e_x(1,t) \m=\m \beta_x(1,t)-
   Y_x(1,t),\crr\disp \beta_{tt}(x,t) \m=\m \beta_{xx}(x,t),
   \crr\disp \beta_x(0,t) \m=\m \widetilde c_1\beta(0,t) + \widetilde
   c_0\beta_t(0,t),\ \ \ \beta(1,t) \m=\m 0,\crr\disp \widetilde{w}
   _{tt}(x,t) \m=\m \widetilde{w}_{xx}(x,t) -(c_1+q)(c_2+q)e^{qx}
   \e(0,t),\crr\disp \widetilde{w}_x(0,t) \m=\m c_2 \widetilde{w}(0,t)
   + (c_1+q)\e(0,t),\ \ \ \widetilde{w}_x(1,t) \m=\m -c_3\widetilde{w}
   _t(1,t),\crr\disp Y_t(x,t) \m=\m -Y_x(x,t),\ \ \ Y(0,t) \m=\m -c_0
   \e(0,t),\end{array} \right.
\end{equation}
in the space $\Xscr=\hline\times H^1(0,1)\times H^1_R(0,1)\times
L^2(0,1)\times\hline$, with the normal inner product.

\begin{lemma} \label{calADAto0}
Suppose that  $c_i>0$, $i=1,2,3$. For any initial
value $(\widetilde\e_0,\widetilde\e_1,Y_0,\beta_0,
\beta_t$, $\widetilde{w}_0,\widetilde{w}_1)\in\Xscr$, with
the compatibility condition $Y_0(0)=-c_0\widetilde\e_0(0)$, the
system {\rm\dref{Lapid}} admits a unique solution $(\e,\e_t,Y,
\beta,\beta_t,\widetilde{w},\widetilde{w}_t)\in C(0,
\infty;\Xscr)$ and there exist two constants $M,\mu>0$ such that
\begin{equation} \label{epYtztw-to0}
   \|(\e(\cdot,t),\e_t(\cdot,t),Y(\cdot,t),\beta(\cdot,t),
   \beta_t(\cdot,t),\widetilde{w}(\cdot,t),\widetilde{w}_t
   (\cdot,t))\|_{\Xscr} \leq Me^{-\mu t}.
\end{equation}
\end{lemma}

\begin{proof} Let $\widetilde\varepsilon(x,t)$ be given by
\dref{tvarep-def}. Introduce a new variable $\eta(x,t)=\widetilde
\e(x,t)-\beta(x,t)$. We convert the system \dref{Lapid} into the
following equivalent system: \vspace{-1mm}
\begin{equation} \label{Tillerson}
   \left\{\begin{array}{l}\eta_{tt}(x,t) \m=\m \eta_{xx}(x,t),\crr
   \disp \eta_x(0,t) \m=\m \widetilde c_1\eta(0,t) + \widetilde c_0
   \eta_t(0,t),\ \ \eta_x(1,t) \m=\m 0,\crr\disp Y_t(x,t) \m=\m
   -Y_x(x,t),\ \ \ Y(0,t) \m=\m -\widetilde c_0 [\eta(0,t)+\beta(0,
   t)],\crr\disp \beta_{tt}(x,t) \m=\m \beta_{xx}(x,t),\crr\disp
   \beta_x(0,t) \m=\m \widetilde c_1\beta(0,t) + \widetilde c_0
   \beta_t(0,t),\ \ \ \beta(1,t) \m=\m 0,\crr\disp \widetilde{w}
   _{tt}(x,t) \m=\m \widetilde{w}_{xx}(x,t) - \frac{(c_1+q)(c_2+q)}
   {1-c_0}e^{qx}[\eta(0,t)+\beta(0,t)],\crr\disp \widetilde{w}_x
   (0,t) \m=\m c_2\widetilde{w}(0,t) + \frac{c_1+q}{1-c_0}
   [\eta(0,t)+\beta(0,t)],\ \ \ \widetilde{w}_x(1,t) \m=\m -c_3
   \widetilde{w}_t(1,t).\end{array} \right.
\end{equation}
We see that the ``$(\eta,\beta)$-part'' of \dref{Tillerson} is
independent of the ``$(Y,\widetilde{w})$-part'' of \dref{Tillerson}.
It is well-known (\cite[Theorem 2.1]{GuoXu2007} and \cite[Theorem 3]
{ChengG1979}) that the subsystem $(\eta,\beta)$ is exponentially
stable, i.e., there exist two constants $M_1,\mu_1>0$ such that
\begin{equation} \label{eps-tz-exp-decay}
   \|(\eta(\cdot,t),\eta_t(\cdot,t),\beta(\cdot,t),\beta_t(\cdot,t)
   \|_{\hline\times H_R^1(0,1)\times L^2(0,1)} \m\leq\m M_1
   e^{-\mu_1t} \m.
\end{equation}
By the Sobolev embedding theorem we have
\begin{equation} \label{ep-tz-exp-decay}
   \begin{array}{l} |\eta(0,t)+\beta(0,t)| \m\leq\m \|\eta(0,t)+\beta
   (0,t)\|_{H^1(0,1)}\crr\disp \m\leq\m 2 \|(\eta(\cdot,t),\eta_t
   (\cdot,t),\beta(\cdot,t),\beta_t(\cdot,t)\|_{\hline\times
   H_R^1(0,1)\times L^2(0,1)} \m\leq\m 2M_1 e^{-\mu_1t} .\end{array}
\end{equation}
We can write the ``$\widetilde{w}$-part'' of \dref{Tillerson} in
operator form as follows:
\begin{equation} \label{varep-abs-tep}
   \frac{\dd}{\dd t} \bbm{\widetilde{w}(\cdot,t)\\ \widetilde{w}_t
   (\cdot,t)} =\m A_{\widetilde{w}} \bbm{\widetilde{w}(\cdot,t)\\
   \widetilde{w}_t(\cdot,t)} + \Bscr_1[\eta(0,t)+\beta(0,t)]+ \Bscr_2
   [\eta(0,t)+\beta(0,t)],
\end{equation}
where the operators $A_{\widetilde{w}}$ is given by \dref{Aw-def} and
$\Bscr_1=(c_1+q)/(1-c_0)(0,-\delta_0)$, $\Bscr_2=-(c_1+q)(c_2+q)/
(1-c_0)(0,-e^{qx})$. Since $A_{\widetilde{w}}$ generates an
exponentially stable operator semigroup $e^{A_{\widetilde{w}}t}$ on
$\hline$ and $\Bscr_1$, $\Bscr_2$ are admissible for this semigroup,
it follows from \dref{ep-tz-exp-decay} and Lemma \ref{Lem-ABu} that
there exist two constants $M_2,\mu_2>0$ such that
\begin{equation} \label{tw-expdecay}
 \|(\widetilde{w}(\cdot,t),\widetilde{w}_t(\cdot,t)\| \m\leq\m
   M_2e^{-\mu_2t} \m.\end{equation}
Next, we claim that the solution of the ``$Y$-part'' of
\dref{Tillerson} is exponentially stable. Define the function
\begin{equation} \label{Y-wp-sol}
   Y(x,t) \m=\m \left\{\begin{array}{ll}\disp -\widetilde c_0
   [\eta(0,t-x)+\beta(0,t-x)], & t\geq x,\crr\disp Y_0(x-t), & x>t.
   \end{array}\right.
\end{equation}
Then it is a straightforward to verify that $Y$ solves the
``$Y$-part'' of \dref{Tillerson}. Based on the proof of the
exponential stability of $W(\cdot,t)$ on ${H^1(0,1)}$ in Lemma
\ref{Lem-hatz-bd}, it suffices to show that there exist two constants
$M_3,\mu_3>0$ such that \vspace{-2mm}
\begin{equation} \label{tepsi-intval-exp-to0}
   \int_0^1 [\eta_t(0,t-x)+\beta_t(0,t-x)]^2 \dd x \m\leq\m
   M_3 e^{-\mu_3t} \m.
\end{equation}
Indeed, define \vspace{-2mm}
$$ \rho(t) \m=\m 2\int_0^1 (x-1)\eta_t(x,t)\eta_x(x,t) \dd x
   + 2\int_0^1 (x-1)\beta_t(x,t)\beta_x(x,t) \dd x.$$
Then
$ |\rho(t)| \m\leq\m \|(\eta(\cdot,t),\eta_t(\cdot,t),\beta(\cdot,
   t),\beta_t(\cdot,t))\|_{\hline\times H_R^1(0,1)\times L^2(0,1)}^2
   \m.$
Computing the derivative of $\rho(t)$ along the solution of
\dref{Tillerson} gives (we suppress the arguments $(x,t)$ that
appear within integrals)
$$ \dot{\rho}(t) = \eta_x^2(0,t)+\eta_t^2(0,t) -\int_0^1 [\eta_x^2
   +\eta_t^2]\dd x + \beta_x^2(0,t) + \beta_t^2(0,t) - \int_0^1
   [\beta_x^2 + \beta_t^2] \dd x$$
$$ \m\ \ \ \geq\m \eta_t^2(0,t) - \int_0^1 [\eta_x^2+\eta_t^2]
   \dd x + \beta_t^2(0,t) - \int_0^1 [\beta_x^2 + \beta_t^2]\dd x,$$
which, combining with \dref{eps-tz-exp-decay}, implies that
$$ \begin{array}{l}\disp \int_{t-1}^t[\eta_s^2(0,s)+ \beta_s^2(0,s)]
   \dd s\crr\disp \leq\m \int_{t-1}^t\|(\eta(\cdot,s),\eta_s(\cdot,s),
   \beta(\cdot,s),\beta_s(\cdot,s))\|_{\hline\times H_R^1(0,1)\times
   L^2(0,1)}^2 \dd s+\rho(t)-\rho(t-1)\crr\disp \leq M_1^2\int_{t-1}^t
   e^{-2\mu_1s} \dd s + M_1^2 e^{-2\mu_1t}+M_1^2e^{-2\mu_1(t-1)}
   \m\leq\m \bigg(\frac{e^{2\mu_1}} {2\mu_1}+e^{2\mu_1}+1\bigg) M_1^2
   e^{-2\mu_1t}.\end{array}$$
On the other hand, since for all $t\geq 1$,
$$ \begin{array}{l}\disp \int_0^1 [\eta_t(0,t-x)+\beta_t(0,t-x)]^2 \dd
   x\crr\disp \leq\m 2\int_0^1 [\eta_t^2(0,t-x) +\beta_t^2(0,t-x)] \dd
   x \m=\m 2\int_{t-1}^t [\eta_s^2(0,s)+\beta_s^2(0,s)] \dd s,
   \end{array}$$
we obtain \dref{tepsi-intval-exp-to0} with \m $M_3=2\big(\frac
{e^{2\mu_1}}{2\mu_1}+e^{2\mu_1}+1\big)M_1^2$ and $\mu_3=2\mu_2$.
Combining with $\e(x,t)=\widetilde\e(x,t)-Y(x,t)$, $Y_t(x,t)=-Y_x
(x,t)$, \dref{eps-tz-exp-decay} and \dref{tw-expdecay}, we get
\dref{epYtztw-to0}. \end{proof}

\begin{remark} {\rm
In the proof of Theorem \ref{calADAto0} below, we introduce the new
variable $\eta(x,t)=\widetilde\e(x,t)-\beta(x,t)$ which is a useful
trick in proving the exponential stability of the subsystem
$\widetilde\e(x,t)$ and the subsystem $\e(x,t)$. This is because we
are not able to prove that $\beta_x(1,t)$ decays exponentially, only
that $\beta_x(1,t)\in L^2[0,\infty)$. So, the exponential stabilities
mentioned cannot follow from Lemma \ref{Lem-ABu}.}
\end{remark}

Now we go back to the closed-loop system \dref{wave-o} under the
feedback \dref{con-out}:
\begin{equation}\label{wave-o-closed-a}
   \left\{\begin{array}{l} w_{tt}(x,t) \m=\m w_{xx}(x,t),\crr\disp
   w_x(0,t) \m=\m -qw(0,t),\crr\disp w_x(1,t) \m=\m -z_x(1,t)+Y_x(1,t)
   -c_3\widehat{w}_t(1,t)-(c_2+q)\widehat{w}(1,t)+f(w(\cdot,t),
   w_t(\cdot,t))\crr\disp\hspace{1.8cm} +\m d(t) - (c_2+q)\int_0^1
   e^{q(1-\xi)}[c_3\widehat{w}_t(\xi,t)+q\widehat{w}(\xi,t)]\dd\xi,
   \crr\disp v_{tt}(x,t) \m=\m v_{xx}(x,t),\crr\disp v_x(0,t) \m=\m
   - qw(0,t) + c_1[v(0,t)-w(0,t)],\crr\disp v_x(1,t) \m=\m -z_x(1,t)
   + Y_x(1,t)-W_x(1,t)-c_3\widehat{w}_t(1,t)-(c_2+q)\widehat{w}(1,t)
   \crr\disp\hspace{1.8cm} -(c_2+q)\int_0^1 e^{q(1-\xi)}\left[ c_3
   \widehat{w}_t(\xi,t)+q\widehat{w}(\xi,t) \right]\dd\xi,
   \end{array}\right.
\end{equation}
\begin{equation} \label{wave-o-closed-b}
   \left\{\begin{array}{l} z_{tt}(x,t) \m=\m z_{xx}(x,t),\crr\disp
   z_x(0,t) \m=\m \widetilde c_1 z(0,t)+\widetilde c_0 z_t(0,t),\
   \ \ \ z(1,t) \m=\m -v(1,t)-W(1,t)+w(1,t),\crr\disp W_t(x,t) \m=\m
   -W_x(x,t),\ \ \ W(0,t) \m=\m -c_0[v(0,t)-w(0,t)],\crr\disp
   \widehat{w}_{tt}(x,t) \m=\m \widehat{w}_{xx}(x,t),\crr\disp
   \widehat{w}_x(0,t) \m=\m -qw(0,t)+c_1[\widehat{w}(0,t)-w(0,t)],
   \crr\disp \widehat{w}_x(1,t) \m=\m -c_3\widehat{w}_t(1,t)-(c_2+q)
   \widehat{w}(1,t) - (c_2+q) \int_0^1 e^{q(1-\xi)}[c_3\widehat{w}_t
   (\xi,t)+q\widehat{w}(\xi,t)] \dd\xi,\crr\disp Y_t(x,t) \m=\m
   -Y_x(x,t),\ \ \ Y(0,t) \m=\m -c_0[\widehat{w}(0,t)-w(0,t)].
   \end{array}\right.
\end{equation}
We consider system \dref{wave-o-closed-a}-\dref{wave-o-closed-b}
in the state space $\Hscr=\hline^3\times H^1(0,1)\times\hline
\times H^1(0,1)$.

\begin{theorem} \label{thm-closed-US-NC}
Suppose that  $c_i>0$, $i=1,2,3$, $f:H^1(0,1)\times
L^2(0,1)\to\mathbb{R}$ is continuous, and $d\in L^\infty[0,\infty)$ or
$d\in L^2[0,\infty)$. For any initial state
$(w_0,w_1,v_0,v_1,z_0,z_1,W_0,\widehat{w}_0,\widehat{w}_1,$ $Y_0)\in
\Hscr$ satisfying the compatibility conditions
$$ -z_0(1)-v_0(1)-W_0(1)+w_0(1) \m=\m 0,\qquad W_0(0)+c_0[v_0(0)-w_0
   (0)] \m=\m 0 \m,$$
$$ Y_0(0)+c_0[\widehat{w}_0(0)-w_0(0)] \m=\m 0 \m,$$
there exists a unique solution to {\rm\dref{wave-o-closed-a}-\dref
{wave-o-closed-b}} such that
$$ (w,w_t,v,v_t,W,z,z_t,\widehat{w}_0,\widehat{w}_t,Y) \m\in\m
   C(0,\infty;\mathcal{H}) \m,$$
\begin{equation} \label{thm-whatwY-decay}
   \|(w(\cdot,t),w_t(\cdot,t), \widehat{w}(\cdot,t),\widehat{w}_t
   (\cdot,t),Y(\cdot,t)) \|_{\mathbb{H}^2\times H^1(0,1)} \m\leq\m
   Me^{-\mu t} \FORALL t\geq 0,
\end{equation}
with some $M,\mu>0$ independent of the initial state, and
\begin{equation} \label{thm-vzW-bd}
   \sup_{t\geq0}\|(v(\cdot,t),v_t(\cdot,t),z(\cdot,t),z_t(\cdot,t),
   W(\cdot,t))\|_{\mathbb{H}^2\times H^1(0,1)} \m<\m \infty \m.
\end{equation}
If we assume further that $f(0,0)=0$ and $d\in L^2[0,\infty)$, then
\begin{equation} \label{thm-vzW-bd0}
   \lim_{t\to\infty} \|(v(\cdot,t),v_t(\cdot,t),z(\cdot,t),z_t
   (\cdot,t),W(\cdot,t))\|_{\hline^2\times H^1(0,1)} \m=\m 0 \m.
\end{equation}
If we assume that $f\equiv0$ and $d\equiv0$, then there exist two
constants $M',\mu'>0$ such that
\begin{equation} \label{thm-vzW-bd-decay}
   \|(v(\cdot,t),v_t(\cdot,t),z(\cdot,t),z_t(\cdot,t),W(\cdot,t))
   \|_{\hline^2\times H^1(0,1)} \m\leq\m M'e^{-\mu' t}\ \ \ \forall
   \m t\geq 0\m.
\end{equation}
\end{theorem}

\begin{proof} Using the variables $\e(x,t),\beta(x,t)$ and
$\widehat{v}(x,t)$ given by \dref{varep-def}, \dref{tildez-def} and
\dref{hatv-def}, respectively, and the invertible transformation
\dref{hatw-tw}, we can rewrite
\dref{wave-o-closed-a}-\dref{wave-o-closed-b} as follows:
\begin{equation} \label{closed-equiv-a}
   \left\{\begin{array}{l} \e_{tt}(x,t) \m=\m \e_{xx}(x,t),\crr\disp
   \e_x(0,t)=c_1\e(0,t),\ \ \ \e_x(1,t)=\beta_x(1,t)-Y_x(1,t),\crr
   \disp \beta_{tt}(x,t) \m=\m \beta_{xx}(x,t),\crr\disp \beta_x(0,t)
   =\m \frac{c_1}{1-c_0}\beta(0,t)+\frac{c_0}{1-c_0}\beta_t(0,t),\ \
   \ \beta(1,t) \m=\m 0,\crr\disp \widetilde{w}_{tt}(x,t) \m=\m
   \widetilde{w}_{xx}(x,t) - (c_1+q)(c_2+q)e^{qx}\e(0,t),\crr\disp
   \widetilde{w}_x(0,t) = c_2\widetilde{w}(0,t)+(c_1+q)\e(0,t),\ \
   \ \widetilde{w}_x(1,t) = -c_3\widetilde{w}_t(1,t),
   \end{array}\right.
\end{equation}
\begin{equation} \label{closed-equiv-b}
   \left\{\begin{array}{l} Y_t(x,t) \m=\m -Y_x(x,t),\ \ \ Y(0,t)
   \m=\m -c_0\e(0,t),\crr\disp \widehat{v}_{tt}(x,t) \m=\m \widehat
   {v}_{xx}(x,t),\crr\disp \widehat{v}_x(0,t) \m=\m c_1\widehat{v}
   (0,t),\crr\disp \widehat{v}_x(1,t) \m=\m -f(w(\cdot,t),w_t(\cdot,
   t))-d(t)-W_x(1,t),\crr\disp W_t(x,t) \m=\m -W_x(x,t),\ \ \ W(0,t)
   =-c_0\widehat{v}(0,t).\end{array}\right.
\end{equation}
It is clear that \dref{wave-o-closed-a}-\dref{wave-o-closed-b} is
well-posed if and only if \dref{closed-equiv-a}-\dref{closed-equiv-b}
is well-posed. We see that the ``$(\e,\beta,\widetilde{w},Y)$-part''
of \dref{closed-equiv-a}-\dref{closed-equiv-b} is independent of the
``$(\widehat{v},W)$-part'' of this system. By Lemma \ref{calADAto0},
there exist two constants $M_1,\mu_1>0$ such that the solution $(\e,
\e_t,Y,\beta,\beta_t,\widetilde{w},\widetilde{w}_t)\in C(0,\infty;
\Xscr)$ satisfies
\begin{equation} \label{proof-epYtztw-to0}
   \|(\e(\cdot,t),\e_t(\cdot,t),Y(\cdot,t),\beta(\cdot,t),\beta_t
   (\cdot,t),\widetilde{w}(\cdot,t),\widetilde{w}_t(\cdot,t))
   \|_\Xscr \leq\m M_1 e^{-\mu_1 t}.
\end{equation}
Owing to the invertibility of the transformation
$$ \begin{pmatrix}w(x,t)\cr w_t(x,t)\cr \widehat{w}(x,t)\cr
   \widehat{w}_t(x,t)\end{pmatrix} = \begin{pmatrix} -I & 0 &
   (I+\pline)^{-1} & 0\cr 0 & -I & 0 & (I+\pline)^{-1}\cr 0 & 0 &
   (I+\pline)^{-1} & 0\cr 0 & 0 & 0 & (I+\pline)^{-1}\end{pmatrix}
   \begin{pmatrix} \e(x,t)\cr \e_t(x,t)\cr \widetilde{w}(x,t)\cr
   \widetilde{w}_t(x,t)\end{pmatrix},$$
where $I+\pline$ is defined by \dref{hatw-tw}, $(w(\cdot,t), w_t
(\cdot,t),\widehat{w}(\cdot,t), \widehat{w}_t(\cdot,t))\in C(0,\infty;
\hline^2)$ is well-defined and satisfies
\begin{equation} \label{th-w-hat-w-pf}
   \|(w(\cdot,t),w_t(\cdot,t),\widehat{w}(\cdot,t),\widehat{w}_t
   (\cdot,t))\|_{\hline^2} \m\leq\m 2[1+\|(I+\mathbb{P})^{-1}\|] M_1
   e^{-\mu_1 t},
\end{equation}
which, combined with \dref{proof-epYtztw-to0}, implies that
\dref{thm-whatwY-decay} holds with $M=3[1+\|(I+\mathbb{P})^{-1}\|]M_1$
and $\mu=\mu_1$. Now we consider the ``$(\widehat{v},W)$-part'':
\begin{equation} \label{Putin}
   \left\{\begin{array}{rl} \widehat{v}_{tt}(x,t) &=\m \widehat{v}
   _{xx}(x,t),\crr\disp \widehat{v}_x(0,t) &=\m c_1\widehat{v}(0,t),
   \ \ \ \widehat{v}_x(1,t) =\m -f(w(\cdot,t),w_t(\cdot,t))-d(t)
   -W_x(1,t),\crr\disp W_t(x,t) &=\m -W_x(x,t),\ \ \ W(0,t) \m=\m
   -c_0\widehat{v}(0,t).\end{array}\right.
\end{equation}
Since  $f:H^1(0,1)\times L^2(0,1)\to\rline$ is continuous and
$(w,\dot{w})$ is bounded, due to the convergence \m $\|(w(\cdot,t),
w_t(\cdot,t))\|_{\hline}\to 0$, we conclude that \m $f(w(\cdot,t),
w_t(\cdot,t))\in L^\infty[0,\infty)$. Since $d\in L^\infty[0,\infty)$
or $d\in L^2[0,\infty)$, it follows from Lemma \ref{Lem-hatz-bd} that
the system \dref{Putin} admits a unique  bounded solution, i.e.,
\begin{equation} \label{th-hatv-W-bd}
   \sup_{t\geq0} \|(\widehat{v}(\cdot,t),\widehat{v}_t(\cdot,t),W
   (\cdot,t))\|_{\mathbb{H}\times H^1(0,1)} \m<\m \infty \m.
\end{equation}
Noting that $W_t(x,t)=-W_x(x,t)$, it follows from \dref{hatv-def},
\dref{tildez-def} and \dref{th-hatv-W-bd} that
$$ \begin{array}{l} \|(v(\cdot,t),v_t(\cdot,t))\|_{\mathbb{H}}
   \m\leq\m \|(\widehat{v}(\cdot,t),\widehat{v}_t(\cdot,t))\|
   _{\hline}+\|(w(\cdot,t),w_t(\cdot,t))\|_{\hline},\crr\disp
   \|(z(\cdot,t),z_t(\cdot,t))\|_{\hline} \m\leq\m \|(\beta(\cdot,t),
   \beta_t(\cdot,t))\|_\hline +\|(\widehat{v}(\cdot,t),\widehat{v}_t
   (\cdot,t))\|_\hline + \|(W(\cdot,t),W_x(\cdot,t))\|_\hline .
   \end{array}$$
The right-hand sides above are finite, which gives \dref{thm-vzW-bd}.

Now suppose that $f(0,0)=0$ and $d\in L^2[0,\infty)$. By
\dref{th-w-hat-w-pf} and the continuity of $f$, we have
$\lim_{t\to\infty}|f(w,w_t)|=0$. By Lemma \ref{Lem-hatz-bd}, we obtain
\begin{equation} \label{lem-v-W-bd0-th}
   \lim_{t\to\infty}\|(\widehat{v}(\cdot,t),\widehat{v}_t(\cdot,t),
   W(\cdot,t))\|_{\mathbb{H}\times H^1(0,1)} \m=\m 0 \m.
\end{equation}
By \dref{proof-epYtztw-to0}, \dref{th-w-hat-w-pf} and
\dref{lem-v-W-bd0-th}, we derive
$$ \begin{array}{l} \|(v(\cdot,t),v_t(\cdot,t))\|_\hline \leq \|
   (\widehat{v}(\cdot,t),\widehat{v}_t(\cdot,t))\|_\hline + \|(w
   (\cdot,t),w_t(\cdot,t))\|_\hline \to 0 \ \mbox{ as }t\to\infty,
   \crr\disp \|(z(\cdot,t),z_t(\cdot,t))\|_\hline \leq \|(\beta
   (\cdot,t),\beta_t(\cdot,t))\|_\hline + \|(\widehat{v}(\cdot,t),
   \widehat{v}_t(\cdot,t))\|_\hline + \|(W(\cdot,t),W_x(\cdot,t))
   \|_\hline \m.\end{array}$$
Next, suppose that $f\equiv 0$ and $d\equiv 0$. It follows from Lemma
\ref{Lem-hatz-bd} that there exist two constants $M_2,\mu_2>0$ such
that for all $t\geq 0$,
\begin{equation} \label{lem-v-W-bd-decay-th}
   \|(\widehat{v}(\cdot,t),\widehat{v}_t(\cdot,t),W(\cdot,t))\|_{\hline
   \times H^1(0,1)} \m\leq\m M_2e^{-\mu_2 t}.
\end{equation}
By \dref{proof-epYtztw-to0}, \dref{th-w-hat-w-pf} and
\dref{lem-v-W-bd-decay-th}, we obtain that for all $t\geq 0$,
$$ \begin{array}{l} \|(v(\cdot,t),v_t(\cdot,t))\|_{\hline} \m\leq\m
   \|(\widehat{v}(\cdot,t),\widehat{v}_t(\cdot,t))\|_{\hline} +
   \|(w(\cdot,t),w_t(\cdot,t))\|_{\hline}\crr\disp\hspace{3.1cm}
   \leq\m M_2e^{-\mu_2 t}+2[1+\|(I+\mathbb{P})^{-1}\|]M_1e^{-\mu_1 t},
   \end{array}$$
$$ \begin{array}{l} \|(z(\cdot,t),z_t(\cdot,t))\|_\hline \m\leq\m
   \|(\beta(\cdot,t),\beta_t(\cdot,t))\|_\hline + \|(\widehat{v}
   (\cdot,t),\widehat{v}_t(\cdot,t))\|_\hline + \|(W(\cdot,t),W_x
   (\cdot,t))\|_\hline\crr\disp\hspace{3.2cm} \leq\m M_1 e^{-\mu_1 t}
   + 3M_2 e^{-\mu_2 t},\end{array}$$
which, combined with \dref{lem-v-W-bd-decay-th}, implies that
\dref{thm-vzW-bd-decay} holds. \end{proof}

\begin{remark} \label{rek-unstablewave} {\rm
The signals $\{w(0,t),w(1,t)\}$ are almost a minimal set of
measurement signals to exponentially stabilize the system
\dref{wave-o}. Indeed, from Theorem \ref{thm-closed-US-NC}, we see
that we can design disturbance estimator and state observer by using
$\{w(0,t) ,w(1,t)\}$ only. Based on this disturbance estimator and
state observer, the system \dref{wave-o} can be exponentially
stabilized by using $\{w(0,t),w(1,t)\}$ only. However,\\ (a). Each of
the observations $\{w(0,t),w(1,t)\}$ alone is not enough for exact
observability, i.e., for any $T>0$, there is no constant $C_T>0$ such
that
$$ \int_0^T w^2(0,s) \dd s \m\geq\m C_T\|w(\cdot,0),w_t(\cdot,0)
   \|_\hline,\ \ \ \int_0^Tw^2(1,s) \dd s \m\geq\m C_T\|
   w(\cdot,0),w_t(\cdot,0)\|_\hline \m.$$
(b). The signal $y(t)=w(1,t)$ is also not enough for exponential
stabilization. Actually, let $f(w,w_t)\equiv 0$, and let $d=q$. Then
the system \dref{wave-o} admits a solution $(w,w_t)=(q(x-1),0)$
which makes the output $y(t)=w(1,t)\equiv 0$.

From (a), (b), $w(0,t)$ seems to be necessary for stabilization. We
leave two open question here: (I): Can we design a state observer for
system \dref{wave-o} using only $y(t)=w(0,t)$? \\ (II): Is
$y(t)=w(0,t)$ enough to make the system \dref{wave-o} stabilizable?}
\end{remark}

\section{An anti-stable wave equation with negative damper}
\label{Sec-US-N} 

In this section we consider the output feedback exponential
stabilization for a new system, where the ``negative spring'' from
\dref{wave-o} is replaced with a ``negative damper'', so that only
the second equation in \dref{wave-o} is changed:
\begin{equation} \label{wave-o-US-N}
   \left\{\begin{array}{rl} w_{tt}(x,t) &=\m w_{xx}(x,t),\crr\disp
   w_x(0,t) &=\m -qw_t(0,t),\crr\disp w_x(1,t) &=\m u(t)+f(w(\cdot,
   t),w_t(\cdot,t))+d(t),\crr\disp w(x,0) &=\m w_0(x),\ \ \ w_t(x,0)
   \m=\m w_1(x),\crr\disp y_{m}(t) &=\m (w(0,t),\m w(1,t)),
   \end{array}\right.
\end{equation}
where $(w,w_t)$ is the state, $u$ is the control input signal, $y_m$
is the output signal, that is, the boundary traces $w(0,t)$ and
$w(1,t)$ are measured. The equations containing the parameter $q>0,
q\neq 1$ creates a destabilizing feedback, it is like the equation
of a damper but with the reversed sign. The function $f:H^1(0,1)
\times L^2(0,1)\to\rline$ is a possibly unknown nonlinear mapping
that represents the internal uncertainty, and $d$ represents the
unknown external disturbance which is only supposed to satisfy $d\in
L^\infty[0,\infty)$.

We consider system \dref{wave-o-US-N} in the state Hilbert space
$\hline=H^1(0,1)\times L^2(0,1)$. The intuitive representation is
as in Figure 1, but with a damper in place of the spring. The
following result is similar to Proposition \ref{Pro01}, and can be
proved along the same lines.

\begin{proposition} \label{Pro01-US}
Suppose that $f:H^1(0,1)\times L^2(0,1)\to\mathbb{R}$ is continuous
with $f(0,0)=0$ and satisfies a global Lipschitz condition in
$H^1(0,1)\times L^2(0,1)$. Then, for any $(w_0,w_1)\in\hline$, $u\in
L^2_{loc}[0,\infty)$, and $d\in L^2_{loc}[0,\infty)$, there exists a
unique global solution to {\rm\dref{wave-o-US-N}} such that
$(w(\cdot,t),\dot{w}(\cdot,t))\in C(0,\infty;\hline)$.
\end{proposition}

\subsection{The disturbance estimator}

We design a disturbance estimator for the system \dref{wave-o-US-N},
that uses the signal $y_{m}(t)=(w(0,t),w(1,t))$, as follows:
\begin{equation} \label{wave-o-v-US-N}
   \left\{\begin{array}{rl} v_{tt}(x,t) \nm &=\m v_{xx}(x,t),\crr
   \disp v_x(0,t) \nm &=\m -qv_t(0,t)+c_1[v(0,t)-w(0,t)],\ \ \
   v_x(1,t) =\m u(t)-W_x(1,t),\crr\disp v(x,0) \nm &=\m v_0(x),\ \ \
   \ \ v_t(x,0) \m=\m v_1(x),\crr\disp z_{tt}(x,t) \nm &=\m z_{xx}
   (x,t),\crr\disp z_x(0,t) \nm &=\m \frac{c_1}{1-c_0} z(0,t)+\frac
   {c_0-q}{1-c_0}z_t(0,t),\ \ \ z(1,t) =\m -v(1,t)-W(1,t)+
   w(1,t),\crr\disp z(x,0) \nm &=\m z_0(x),\ \ \ z_t(x,0) \m=\m z_1
   (x),\crr\disp W_t(x,t) \nm &=\m -W_x(x,t),\crr\disp W(0,t) \nm
   &=\m -c_0[v(0,t)-w(0,t)],\qquad W(x,0) \m=\m W_0(x),
   \end{array}\right.
\end{equation}
where $c_0$ and $c_1$ are two design parameters so that $\frac{c_1}
{1-c_0}>0$ and $\frac{c_0-q}{1-c_0}>0$. The initial state of the
disturbance estimator \dref{wave-o-v-US-N} is \m $(v_0,v_1.z_0,z_1,
W_0)\in \mathbb{H}^2\times H^1(0,1)$. It is clear that the above
disturbance estimator receives as inputs the control input $u$ of
the original system and the two measurement signals $w(0,t)$ and $
w(1,t)$. The ``$(v,W)$-subsystem'' is an auxiliary system which is
used to separate the total disturbance from the original system
\dref{wave-o-US-N} to an exponential system. Indeed, let
\begin{equation} \label{hatv-def-US-N}
   \widehat{v}(x,t) \m=\m v(x,t)-w(x,t) \m.
\end{equation}
Then it is easy to verify that  $(\widehat{v}(x,t),W(x,t))$ satisfies
\begin{equation} \label{wave-o-hatv-P-US-N}
\left\{\begin{array}{l}
\widehat{v}_{tt}(x,t)=\widehat{v}_{xx}(x,t),  \crr\disp
\widehat{v}_x(0,t)=-q\widehat{v}_t(0,t)+c_1\widehat{v}(0,t),  \crr\disp
\widehat{v}_x(1,t)=-f(w(\cdot,t),w_t(\cdot,t))-d(t)-W_x(1,t), \crr\disp
W_t(x,t)=-W_x(x,t), \;\;
W(0,t)=-c_0\widehat{v}(0,t).
\end{array}\right.
\end{equation}
It is seen that the inhomogeneous part of \dref{wave-o-hatv-P-US-N}
is just the total disturbance.

\begin{lemma} \label{Lem-hatz-bd-US-N}
Suppose that $\frac{c_1}{1-c_0}>0$, $\frac{c_0-q}{1-c_0}>0$; $d\in
L^\infty[0,\infty)$, (or $d\in L^2[0,\infty)$), $f:H^1(0,1)\times L^2
(0,1)\to\rline$ is continuous and that {\rm\dref{wave-o-US-N}} admits a
unique solution $(w,\dot{w})\in C(0,\infty;\hline)$ which is bounded.
For any initial value $(\widehat{v}_0,\widehat{v}_1,W_0)\in\hline\times
H^1(0,1)$ with the compatibility condition $W_0(0)=-c_0\widehat{v}_0
(0)$, there exists a unique solution $(\widehat{v},\widehat{v}_t,W)\in
C(0,\infty;\hline\times H^1(0,1))$ to {\rm\dref{wave-o-hatv-P-US-N}} such
that
\begin{equation} \label{lem-v-W-bd-US-N}
   \sup_{t\geq 0}\|(\widehat{v}(\cdot,t),\widehat{v}_t(\cdot,t),
   W(\cdot,t))\|_{\hline\times H^1(0,1)} < \infty \m.
\end{equation}
If we assume further that  $\lim_{t\to\infty}|f(w,w_t)|=0$ and $d\in
L^2[0,\infty)$, then
\begin{equation} \label{lem-v-W-bd0-US-N}
   \lim_{t\to\infty} \|(\widehat{v}(\cdot,t),\widehat{v}_t(\cdot,t),
   W(\cdot,t))\|_{\hline\times H^1(0,1)} \m=\m 0 \m.
\end{equation}
If we assume that $f\equiv 0$ and $d\equiv 0$, then there exist two
constants $M,\mu>0$ such that
\begin{equation} \label{lem-v-W-bd-decay-US-N}
   \|(\widehat{v}(\cdot,t),\widehat{v}_t(\cdot,t),W(\cdot,t))
   \|_{\hline\times H^1(0,1)} \m\leq\m Me^{-\mu t}\ \ \ \forall
   \m t\geq 0 \m.
\end{equation}
\end{lemma}

\begin{proof} First we introduce a new variable $p(x,t)=
-\widehat{v}(x,t)-W(x,t)$, then $(p(x,t),W(x,t))$ satisfies
\begin{equation} \label{wave-o-hatv-US-N}
   \left\{\begin{array}{rl} p_{tt}(x,t) \nm &=\m
   p_{xx}(x,t),\crr\disp p_x(0,t) \nm &=\m
   \frac{c_1}{1-c_0}p(0,t) + \frac{c_0-q}{1-c_0} \m
   p_t(0,t),\crr\disp p_x(1,t) \nm &=\m f(w(\cdot,t),w_t(\cdot,t))+d(t),\crr\disp W_t(x,t) \nm &=\m -W_x
   (x,t),\ \ \ \ W(0,t) \m=\m \frac{c_0}{1-c_0}p(0,t),
   \end{array}\right.
\end{equation}
with the initial state \m $p(x,0)=-\widehat{v}(x,0)
-W(x,0)$, \m $p_t(x,0)=-\widehat{v}_t(x,0)+W_x(x,0)$ and
\m $W(x,0)=W(x,0)$.
Comparing \dref{wave-o-hatv-US-N} with \dref{wave-o-hatv}, it is seen
that \dref{wave-o-hatv-US-N} is exactly the same as the system
\dref{wave-o-hatv} by replacing $\frac{c_0}{1-c_0}$ with
$\frac{c_0-q}{1-c_0}$. Thus, the rest of the proof of this lemma is
exactly the same as for Lemma \ref{Lem-hatz-bd}. \end{proof}

Let \vspace{-2mm}
\begin{equation} \label{tildez-def-US-N}
   \beta(x,t) \m=\m z(x,t)-p(x,t).
\end{equation}
Then we can see that $\beta(x,t)$ is governed by
\begin{equation} \label{wave-o-tildz-US-N}
   \left\{\begin{array}{l} \beta_{tt}(x,t) \m=\m
   \beta_{xx}(x,t),\crr\disp \beta_x(0,t) \m=\m
   \frac{c_1}{1-c_0}\beta(0,t)+\frac{c_0-q}{1-c_0}
   \beta_t(0,t),\ \ \ \beta(1,t) \m=\m 0,
\end{array}\right.
\end{equation}
which is exactly the same as the system \dref{wave-o-tildz} by
replacing $\frac{c_0}{1-c_0}$ with $\frac{c_0-q}{1-c_0}$. Therefore,
it follows from Lemma \ref{dis-est} and Remark \ref{Rem-k-est-d} that
$z_x(1,t)$ can be regarded as an estimate of the total disturbance
$F(t)=f(w(\cdot,t),w_t(\cdot,t))+d(t)$, that is, $z_x(1,t)\approx
F(t)$.

\subsection{Controller and observer design}

In this subsection we investigate the following state observer for
the system \dref{wave-o-US-N}:
\begin{equation} \label{wave-o-obser-US-N}
   \left\{\begin{array}{rl} \widehat{w}_{tt}(x,t) \nm &=\m \widehat
   {w}_{xx}(x,t),\crr\disp \widehat{w}_x(0,t) \nm &=\m -q\widehat{w}
   _t(0,t) + c_1[\widehat{w}(0,t)-w(0,t)],\crr\disp \widehat{w}_x(1,t)
   \nm &=\m u(t)+z_x(1,t)-Y_x(1,t),\crr\disp \widehat{w}(x,0) \nm
   &=\m \widehat{w}_0(x),\ \ \ \widehat{w}_t(x,0) \m=\m \widehat{w}_1
   (x),\crr\disp Y_t(x,t) \nm &=\m -Y_x(x,t),\crr\disp Y(0,t) \nm
   &=\m -c_0[\widehat{w}(0,t)-w(0,t)],\ \ \ Y(x,0) \m=\m Y_0(x),
   \end{array}\right.
\end{equation}
where $x\in[0,1]$, $t\geq 0$ and $c_1,\m c_2$ are design parameters
that are the same as in \dref{wave-o-v-US-N}. Here $z_x(1,t)$ plays
the role of total disturbance. To show the asymptotic convergence of
the observer above, define
\begin{equation} \label{varep-def-US-N}
   \e(x,t) \m=\m \widehat{w}(x,t)-w(x,t).
\end{equation}
Then it is easy to see that $(\varepsilon(x,t),Y(x,t))$ is
governed by
\begin{equation} \label{wave-o-obser-err-P-US-N}
   \left\{\begin{array}{l} \e_{tt}(x,t) \m=\m \e_{xx}(x,t),\crr
   \disp \e_x(0,t) \m=\m -q\e_t(0,t) + c_1\e(0,t),\ \
   \e_x(1,t) \m=\m \beta_x(1,t)-Y_x(1,t),\crr\disp
   Y_t(x,t) \m=\m -Y_x(x,t),\ \ \ Y(0,t) \m=\m -c_0\e(0,t).
   \end{array}\right.
\end{equation}

\begin{lemma} \label{lem-varepto0-US-N}
Suppose that $\frac{c_0-q}{1-c_0}>0$, $\frac{c_1}{1-c_0}>0$ and the
signal $\widetilde{z}_x(1,t)$ is determinated by system
{\rm\dref{wave-o-tildz}}. Then for any initial state $(\e(\cdot,0),\e_t
(\cdot,0),Y(\cdot,0))\in\hline\times H^1(0,1)$ with the
compatibility condition $Y(0,0)=-c_0\e(0,0)$, there exists a unique
solution to {\rm\dref{wave-o-obser-err-P-US-N}} such that
$(\e,\e_t,Y)\in C(0,\infty;\hline\times H^1(0,1))$ satisfying
\begin{equation} \label{varepsi-to0-US-N}
   \lim_{t\to\infty} \|(\e(\cdot,t),\e_t(\cdot,t),Y(\cdot,t))\|
   _{\hline\times H^1(0,1)} \m=\m 0 \m.
\end{equation}
\end{lemma}

\begin{proof} Introduce a new variable $\widetilde\e(x,t)=\e(x,t)+
Y(x,t)$. Then $(\widetilde\e(x,t),Y(x,t))$ satisfies
\begin{equation} \label{wave-o-obser-err-US-N}
   \left\{\begin{array}{l} \widetilde\varepsilon_{tt}(x,t) =
   \widetilde\e_{xx}(x,t),\crr\disp \widetilde\e_x(0,t) =
   \widetilde c_1\widetilde\e(0,t) + \frac{c_0-q}{1-c_0}
   \widetilde\e_t(0,t),\ \ \ \widetilde\e_x(1,t) = \beta
   _x(1,t),\crr\disp Y_t(x,t) = -Y_x(x,t),\ \ \ Y(0,t)=-\widetilde
   c_0\widetilde\e(0,t),\end{array}\right.
\end{equation}
with the initial state $\widetilde\e(x,0)=\e(x,0)+Y(x,0),\;
\widetilde\e_t(x,0)=\e_t(x,0)-Y_x(x,0),\;Y(x,0)=Y(x,0)$. Comparing
\dref{wave-o-obser-err-US-N} with \dref{Boris}, we see that
\dref{wave-o-obser-err-US-N} is exactly the same as the system
\dref{Boris} by replacing $\widetilde c_0$  with $\frac{c_0-q}
{1-c_0}$ for ``$\widetilde\e$-part''. Thus, according to the proof
of Lemma \ref{lem-varepto0}, we can conclude that
\dref{wave-o-obser-err-P-US-N} admits a unique solution satisfying
\dref{varepsi-to0-US-N}.
\end{proof}

By Lemma \ref{lem-varepto0-US-N}, \dref{wave-o-obser-US-N} is indeed
an observer of \dref{wave-o-US-N}. To find a stabilizing control law
for system \dref{wave-o-US-N}, we introduce the following auxiliary
system (here $t\geq 0$ and $x\in[0,1]$):
$$ \left\{\begin{array}{l}
   Z_t(x,t) \m=\m -Z_x(x,t),\crr\disp Z(0,t) \m=\m -c_2\widehat{w}
   (0,t),\qquad Z(x,0) \m=\m Z_0(x). \end{array}\right.$$
Now we introduce the new variable \m $\widetilde{w}(x,t)=\widehat{w}
(x,t)+Z(x,t)$. Then $(\widetilde{w},Z)$ satisfies
\begin{equation} \label{wave-o-obser-trans-tw-US-N}
   \left\{\begin{array}{l} \widetilde{w}_{tt}(x,t) \m=\m \widetilde
   {w}_{xx}(x,t),\crr\disp \widetilde{w}_x(0,t) \m=\m \frac{c_2-q}
   {1-c_2}\widetilde{w}_t(0,t)+c_1\e(0,t),\crr\disp \widetilde{w}
   _x(1,t) \m=\m u(t)+z_x(1,t)-Y_x(1,t)+Z_x(1,t),\crr\disp
   Z_t(x,t) \m=\m -Z_x(x,t),\ \ \ Z(0,t) \m=\m -\frac{c_2}{1-c_2}
   \widetilde{w}(0,t). \end{array}\right.
\end{equation}
We see that the exponential stability of system
\dref{wave-o-obser-trans-tw-US-N} is equivalent to the exponential
stability of \dref{wave-o-obser-US-N}. We propose the following
observer-based feedback controller:
\begin{equation} \label{con-out-US-N}
   u(t) \m=\m -c_3\widehat{w}(1,t)-c_3Z(1,t)-z_x(1,t)+Y_x(1,t)
   -Z_x(1,t).
\end{equation}
The closed-loop system formed by \dref{wave-o-obser-trans-tw-US-N}
with the controller \dref{con-out-US-N} becomes \vspace{-1mm}
\begin{equation} \label{wave-o-obser-trans-tw-closed-US-N}
   \left\{\begin{array}{l} \widetilde{w}_{tt}(x,t) \m=\m \widetilde
   {w}_{xx}(x,t),\crr\disp \widetilde{w}_x(0,t) = \frac{c_2-q}
   {1-c_2}\widetilde{w}_t(0,t)+c_1\varepsilon(0,t), \;\;
   \widetilde{w}_x(1,t) \m=\m -c_3\widetilde{w}(1,t),\crr\disp
   Z_t(x,t) \m=\m -Z_x(x,t),\ \ \ Z(0,t) \m=\m -\frac{c_2}{1-c_2}
   \widetilde{w}(0,t). \end{array}\right.
\end{equation}
The closed-loop of observer \dref{wave-o-obser-US-N} corresponding
to controller \dref{con-out-US-N} becomes
\begin{equation} \label{wave-o-obser-trans-hatw-closed-US-N}
   \left\{\begin{array}{l} \widehat{w}_{tt}(x,t)=\widehat{w}_{xx}
   (x,t),\crr\disp \widehat{w}_x(0,t) = -q\widehat{w}_t(0,t) + c_1
   [\widehat{w}(0,t)-w(0,t)],\crr\disp \widehat{w}_x(1,t) = -c_3
   \widehat{w}(1,t)-c_3Z(1,t)-Z_x(1,t),\crr\disp Y_t(x,t) = -Y_x
   (x,t),\ \ \ Y(0,t) = -c_0[\widehat{w}(0,t)-w(0,t)],\crr\disp
   Z_t(x,t) = -Z_x(x,t),\ \ \ Z(0,t) = -\frac{c_2}{1-c_2}
   \widetilde{w}(0,t).\end{array}\right.
\end{equation}
To show the exponential stability of system
\dref{wave-o-obser-trans-tw-US-N} under the feedback
\dref{con-out-US-N}, we consider the overall system
\dref{wave-o-obser-err-P-US-N}, \dref{wave-o-tildz-US-N} and
\dref{wave-o-obser-trans-tw-closed-US-N} described by
\begin{equation} \label{Obama}
   \left\{\begin{array}{l} \e_{tt}(x,t) \m=\m \e_{xx}(x,t),\crr\disp
   \e_x(0,t) \m=\m -q\e_t(0,t)+c_1\e(0,t),\ \ \ \e_x(1,t) \m=\m
   \beta_x(1,t)-Y_x(1,t),\crr\disp Y_t(x,t) \m=\m -Y_x(x,t),\ \ \
   Y(0,t) \m=\m -c_0\e(0,t),\crr\disp \beta_{tt}(x,t) \m=\m \beta
   _{xx}(x,t),\crr\disp \beta_x(0,t) \m=\m \frac{c_1}{1-c_0}\beta
   (0,t) + \frac{c_0-q}{1-c_0}\beta_t(0,t),\ \ \ \beta(1,t) \m=\m 0,
   \crr\disp \widetilde{w}_{tt}(x,t) \m=\m \widetilde{w}_{xx}(x,t),
   \crr\disp \widetilde{w}_x(0,t) \m=\m \frac{c_2-q}{1-c_2}
   \widetilde{w}_t(0,t) + c_1\varepsilon(0,t),\ \ \ \widetilde{w}_x
   (1,t) \m=\m -c_3 \widetilde{w}(1,t),\crr\disp Z_t(x,t) \m=\m -Z_x
   (x,t),\ \ \ Z(0,t) \m=\m -\frac{c_2}{1-c_2}\widetilde{w}(0,t),
   \end{array}\right.
\end{equation}
in the state space $\Xscr=\hline\times H^1(0,1)\times H^1_R(0,1)
\times L^2(0,1)\times\hline\times H^1(0,1)$.

\begin{theorem} \label{calADAto0-US-N}
Suppose that $\frac{c_1}{1-c_0}>0$, $\frac{c_0-q}{1-c_0}>0$, $\frac
{c_2-q}{1-c_2}>0$ and $c_3>0$. For any initial state $(\widetilde
\e_0,\widetilde\e_1,Y_0,\beta_0,\beta_t$, $\widetilde{w}_0,
\widetilde{w}_1,Z)\in\Xscr$, with the compatibility conditions
$Y_0(0)=-c_0\widetilde\e_0(0)$, $Z_0(0)=-\frac{c_2}{1-c_2}
\widetilde{w}_0(0)$, the system {\rm\dref{Obama}} admits a unique
solution $(\e,\e_t,Y,\beta,\beta_t,\widetilde{w},\widetilde{w}_t,Z)
\in C(0,\infty;\Xscr)$ and there exist two constants $M,\mu>0$ such
that
\begin{equation} \label{epYtztw-to0-US-N}
   \|(\e(\cdot,t),\e_t(\cdot,t),Y(\cdot,t),\beta(\cdot,t),
   \beta_t(\cdot,t),\widetilde{w}(\cdot,t),\widetilde{w}_t
   (\cdot,t),Z(\cdot,t))\|_{\Xscr} \leq Me^{-\mu t}.
\end{equation}
\end{theorem}

\begin{proof} We see that the ``$(\e,Y,\beta)$-part'' of \dref{Obama}
is independent of the ``$(\widetilde{w},Z)$-part'' of \dref{Obama}. We
first consider the ``$(\e,Y,\beta)$-part'' of \dref{Obama}. Denote
$\eta(x,t)=\e(x,t)+Y(x,t)-\beta(x,t) $. It is easy to check that
$(\eta(x,t),Y(x,t),\beta(x,t))$ satisfies
\begin{equation} \label{Theresa_May}
   \left\{\begin{array}{l} \eta_{tt}(x,t) = \eta_{xx}(x,t),\crr\disp
   \eta_x(0,t) = \frac{c_1}{1-c_0}\eta(0,t)+\frac{c_0-q}{1-c_0}\eta_t
   (0,t),\ \ \eta_x(1,t) = 0,\crr\disp Y_t(x,t) = -Y_x(x,t),\ \
   Y(0,t) = -\frac{c_0}{1-c_0}[\eta(0,t)+\beta(0,t)],\crr
   \disp \beta_{tt}(x,t) = \beta_{xx}(x,t),\crr\disp
   \beta_x(0,t)=\frac{c_1}{1-c_0}\beta(0,t) +
   \frac{c_0-q}{1-c_0}\beta_t(0,t),\ \ \beta(1,t) =
   0 \m.\end{array}\right.
\end{equation}
Comparing \dref{Theresa_May} with \dref{Tillerson} and noting that
$\frac{c_1}{1-c_0}>0$, $\frac{c_0-q}{1-c_0}>0$, the system
\dref{Theresa_May} is exactly the same as the system \dref{Tillerson}
after 
replacing $\frac{c_0}{1-c_0}>0$ with $\frac{c_0-q}{1-c_0}>0$. Thus, by Theorem
\ref{calADAto0}, we can conclude that \dref{Theresa_May} admits a
unique solution and there exist two constants $M_1,\mu_1>0$ such that
\begin{equation} \label{epYtztw-part-to0-US-N}
   \|(\e(\cdot,t),\e_t(\cdot,t),\beta(\cdot,t),
   \beta_t(\cdot,t),Y(\cdot,t))\|_{\hline\times H^1_R(0,1)
   \times L^2(0,1)\times H^1(0,1)} \leq M_1e^{-\mu_1 t}.
\end{equation}
Now, we consider the ``$(\widetilde{w},Z)$-part'' of \dref{Obama}
which reads as
\begin{equation} \label{th-pf-wZ-part-US-N}
   \left\{\begin{array}{l} \widetilde{w}_{tt}(x,t) =\widetilde{w}
   _{xx}(x,t),\crr\disp \widetilde{w}_x(0,t)=\frac{c_2-q}{1-c_2}
   \widetilde{w}_t(0,t)+c_1\varepsilon(0,t),\ \ \widetilde{w}_x(1,t)
   = -c_3\widetilde{w}(1,t),\crr\disp Z_t(x,t) = -Z_x(x,t),\ \
   Z(0,t) = -\frac{c_2}{1-c_2}\widetilde{w}(0,t),\end{array}\right.
\end{equation}
By Sobolev embedding theorem and \dref{epYtztw-part-to0-US-N}, we have
\begin{equation} \label{ep-tz-exp-decay-US-N}
   \begin{array}{l} |\e(0,t)| \leq \|\e(0,t)\|_{H^1(0,1)} \leq
   \|(\e(\cdot,t),\e_t(\cdot,t)\|_{\hline} \leq M_1e^{-\mu_1t}.
   \end{array}
\end{equation}
We can write ``$\widetilde{w}$-part'' of \dref{th-pf-wZ-part-US-N} as
$$ \frac{\dd}{\dd t}(\widetilde{w}(\cdot,t),\widetilde{w}_t
   (\cdot,t)) = A_0(\widetilde{w}(\cdot,t),\widetilde{w}_t
   (\cdot,t)) + B_0\e(0,t),$$
where the operators $A_0$ and $B_0$ are given by
\begin{equation} \label{A0-def}
   \left\{ \begin{array}{l}\disp A_0(\phi,\psi) \m=\m (\psi,\phi'')
   \FORALL (\phi,\psi)\in\Dscr(A_0),\crr\disp \Dscr(A_0) = \bigg\{
   (\phi,\psi)\in H^2(0,1)\times H^1(0,1)\ |\ \phi'(0)=\frac{c_2-q}
   {1-c_2}\psi(0),\;\phi'(1)=-c_3\phi(1)\bigg\},\end{array}\right.
\end{equation}
and $B_0=c_1(0,-\delta_0)$. It is well known (\cite[Proposition 2]
{JFF2013tac}) that $A_{0}$ generates an exponential stable operator
semigroup $e^{A_{0}t}$ on $\mathbb{H}$ and $B_0$ is admissible for
$e^{A_{0}t}$. It follows from \dref{ep-tz-exp-decay-US-N} and Lemma
\ref{Lem-ABu} that there exist two constant $M_2,\mu_2>0$ such that
\begin{equation} \label{eps-tz-exp-decay-US-N}
   \|(\widetilde{w}(\cdot,t),\widetilde{w}_t(\cdot,t)\| \m\leq\m
   M_2e^{-\mu_2t} \m.
\end{equation}
Next, we claim that the solution of ``$Z$-part'' of
 \dref{th-pf-wZ-part-US-N} is exponentially stable. Set
\begin{equation} \label{Z-wp-sol-US-N}
   Z(x,t) \m=\m \left\{\begin{array}{ll}\disp -\frac{c_2}{1-c_2}
   \widetilde{w}(0,t-x),&t\geq x,\crr\disp Z_0(x-t),&x>t \m.
   \end{array}\right.
\end{equation}
Then a direct computation shows that $Z(x,t)$
solves ``$Z$-part'' of \dref{th-pf-wZ-part-US-N}.  Thus, to show the
exponentially stability of ``$Z$-part'' of \dref{th-pf-wZ-part-US-N},
it suffices to prove that there exist two constants $M_3,\mu_3>0$ such
that
\begin{equation} \label{tw-intval-exp-to0-US-N}
   \int_0^1\widetilde{w}_t^2(0,t-x) \dd x \m\leq\m M_3e^{-\mu_3t} \m.
\end{equation}
{ 
Indeed, \dref{tw-intval-exp-to0-US-N}  can be proved
   by defining  $\rho(t) \m=\m 2\int_0^1 (x-1)\widetilde{w}_t(x,t)\widetilde{w}_x
   (x,t) \dd x \m.$ Since the proof of \dref{tw-intval-exp-to0-US-N} is very similar
   to the proof of \dref{tepsi-intval-exp-to0}, we omit the details.}
Combining \dref{epYtztw-part-to0-US-N},
\dref{eps-tz-exp-decay-US-N}, \dref{Z-wp-sol-US-N} and
the exponential stability of $Z(\cdot,t)$ on $H^1(0,1)$,
 we get   \dref{epYtztw-to0-US-N}.
  \end{proof}

\subsection{Well-posedness and exponential stability of the
            closed-loop system}

We go back to the closed-loop system of \dref{wave-o-US-N} under the
feedback \dref{con-out-US-N}:
\begin{equation} \label{Macron-a}
   \left\{\begin{array}{rl} w_{tt}(x,t) &=\m w_{xx}(x,t),\crr\disp
   w_x(0,t) &=\m -qw_t(0,t),\crr\disp w_x(1,t) &=\m -c_3\widehat
   {w}(1,t)-c_3Z(1,t)-z_x(1,t)+Y_x(1,t)\crr\disp\hspace{1.8cm}
   \m & \ \ \ \ \ -Z_x(1,t)+f(w(\cdot,t),w_t(\cdot,t))+d(t),\crr\disp
   v_{tt}(x,t) &=\m v_{xx}(x,t),\crr\disp v_x(0,t) &=\m -qv_t(0,t) +
   c_1[v(0,t) - w(0,t)],\crr \disp v_x(1,t) &=\m -c_3\widehat{w}(1,t)
   -c_3 Z(1,t)-z_x(1,t)+Y_x(1,t)\crr\disp\hspace{1.8cm} \m & \ \ \ \
   \ -Z_x(1,t)-W_x(1,t), \end{array}\right.
\end{equation}
\begin{equation} \label{Macron-b}
   \left\{\begin{array}{rl} z_{tt}(x,t) \nm &=\m z_{xx}(x,t),\crr
   \disp z_x(0,t) \nm &=\m \frac{c_1}{1-c_0}z(0,t)+\frac{c_0-q}
   {1-c_0} z_t(0,t),\ \ \ z(1,t) =\m -v(1,t)-W(1,t)+w(1,t),\crr\disp
   \widehat{w}_{tt}(x,t) \nm &=\m \widehat{w}_{xx}(x,t),\crr\disp
   \widehat{w}_x(0,t) \nm &=\m -q\widehat{w}_t(0,t)+c_1[\widehat{w}
   (0,t)-w(0,t)],\crr\disp \widehat{w}_x(1,t) \nm &=\m -c_3\widehat
   {w}(1,t)-c_3Z(1,t)-Z_x(1,t),\crr\disp W_t(x,t) \nm &=\m -W_x(x,t),
   \ \ W(0,t) \m=\m -c_0[v(0,t)-w(0,t)],\crr\disp Y_t(x,t) \nm &=\m
   -Y_x(x,t),\ \ Y(0,t) \m=\m -c_0[\widehat{w}(0,t)-w(0,t)],\crr\disp
   Z_t(x,t) \nm &=\m -Z_x(x,t),\ \ Z(0,t) \m=\m-c_2\widehat{w}(0,t)
   \m. \end{array}\right.
\end{equation}
We consider the system \dref{Macron-a}-\dref{Macron-b} in the state
space $\mathscr{H}=\hline^3\times H^1(0,1)\times\hline\times
[H^1(0,1)]^2$.

\begin{theorem} \label{Thm-closed-USN}
Suppose that $\frac{c_1}{1-c_0}>0$, $\frac{c_0-q}{1-c_0}>0$,
$\frac{c_2-q}{1-c_2}>0$ and $c_3>0$. Suppose that $f:\hline\to\rline$
is continuous, and $d\in L^\infty[0,\infty)$ or $d\in L^2[0,\infty)$.
For any initial state $(w_0,w_1,v_0,v_1,z_0,z_1,W_0,\widehat{w}_0,
\widehat{w}_1,Y_0,Z_0)\in\mathscr{H}$ with the compatibility
conditions
$$ \begin{array}{l} -z_0(1)-v_0(1)-W_0(1)+w_0(1) \m=\m 0 \m,\ \ \
   Z_0(0)+c_2\widehat{w}_0(0) \m=\m 0 \m,\crr\disp W_0(0)+c_0[v_0(0)
   -w_0(0)] \m=\m 0,\ \ \ Y_0(0)+c_0[\widehat{w}_0(0)-w_0(0)] \m=\m
   0 \m,\end{array}$$
there exists a unique solution to {\rm\dref{Macron-a}-\dref{Macron-b}}
such that\\ $(w,w_t,v,v_t,z,z_t,W,\widehat{w}_0,\widehat{w}_t,Y,Z)\in
C(0,\infty;\mathscr{H})$ satisfies
\begin{equation} \label{thm-whatwY-decay-US-N}
   \|(w(\cdot,t),w_t(\cdot,t),\widehat{w}(\cdot,t),\widehat{w}_t
   (\cdot,t),Y(\cdot,t),Z(\cdot,t))\|_{\mathbb{H}^2\times
   [H^1(0,1)]^2} \m\leq\m Me^{-\mu t},\ \ t\geq 0,
\end{equation}
with some $M,\mu>0$, and
\begin{equation} \label{thm-vzW-bd-US-N}
   \sup_{t\geq0}\|(v(\cdot,t),v_t(\cdot,t),z(\cdot,t),z_t(\cdot,t),
   W(\cdot,t))\|_{\hline^2\times H^1(0,1)} \m<\m \infty \m.
\end{equation}
If we assume further that  $f(0,0)=0$ and $d\in L^2[0,\infty)$, then
\begin{equation} \label{thm-vzW-bd0-US-N}
   \lim_{t\to\infty} \|(v(\cdot,t),v_t(\cdot,t),z(\cdot,t),z_t
   (\cdot,t),W(\cdot,t))\|_{\hline^2\times H^1(0,1)} \m=\m 0 \m.
\end{equation}
If we assume that $f\equiv 0$ and $d\equiv 0$, then there exist two
constants $M',\mu'>0$ such that
\begin{equation} \label{thm-vzW-bd-decay-US-N}
   \|(v(\cdot,t),v_t(\cdot,t),z(\cdot,t),z_t(\cdot,t),W(\cdot,t))
   \|_{\hline^2\times H^1(0,1)} \m\leq\m M'e^{-\mu' t},\ \ t\geq 0 \m.
\end{equation}
\end{theorem}

\begin{proof}
Using the variables $\e(x,t),\beta(x,t)$ and $\widehat{v}
(x,t)$ given by \dref{varep-def-US-N}, \dref{tildez-def-US-N} and
 \dref{hatv-def-US-N}, respectively, and the invertible
transformation $\widetilde{w}(x,t)=\widehat{w}(x,t)+Z(x,t)$,
we can write a system equivalent to \dref{Macron-a}-\dref{Macron-b}
as follows:
\begin{equation} \label{closed-equiv-US-Na}
   \left\{\begin{array}{l} \e_{tt}(x,t) \m=\m \e_{xx}(x,t),\crr
   \disp \e_x(0,t) \m=\m -q\e_t(0,t)+c_1\e(0,t),\ \ \ \e_x(1,t)
   \m=\m \beta_x(1,t)-Y_x(1,t),\crr\disp Y_t(x,t) \m=\m
   -Y_x(x,t),\ \ \ Y(0,t) \m=\m -c_0\e(0,t),\crr\disp
   \beta_{tt}(x,t) \m=\m \beta_{xx}(x,t),\crr\disp
   \beta_x(0,t) \m=\m \frac{c_1}{1-c_0}\beta(0,t)+\frac
   {c_0-q}{1-c_0}\beta_t(0,t),\ \ \ \beta(1,t)
   \m=\m 0,\end{array}\right.
\end{equation}
\begin{equation} \label{closed-equiv-US-Nb}
   \left\{\begin{array}{l} \widetilde{w}_{tt}(x,t) \m=\m
   \widetilde{w}_{xx}(x,t),\crr\disp \widetilde{w}_x(0,t) \m=\m
   \frac{c_2-q}{1-c_2}\widetilde{w}_t(0,t) + c_1\e(0,t),\ \ \
   \widetilde{w}_x(1,t)=-c_3\widetilde{w}(1,t),\crr\disp Z_t(x,t)
   \m=\m -Z_x(x,t),\ \ \ Z(0,t) \m=\m -\frac{c_2}{1-c_2}
   \widetilde{w}(0,t),\crr\disp \widehat{v}_{tt}(x,t) \m=\m
   \widehat{v}_{xx}(x,t),\crr\disp \widehat{v}_x(0,t) \m=\m -q
   \widehat{v}_t(0,t)+c_1\widehat{v}(0,t),\crr\disp \widehat{v}_x

   (1,t) \m=\m -f(w(\cdot,t),w_t(\cdot,t))-d(t)-W_x(1,t),\crr
   \disp W_t(x,t) \m=\m -W_x(x,t),\ \ \ W(0,t) \m=\m -c_0
   \widehat{v}(0,t).\end{array}\right.
\end{equation}
We see that the ``$(\e,\beta,\widetilde{w},Y)$-part'' of
\dref{closed-equiv-US-Na}-\dref{closed-equiv-US-Nb} is independent of
the ``$(\widehat{v},W)$-part'' of \dref{closed-equiv-US-Na}-%
\dref{closed-equiv-US-Nb}. By Theorem \ref{calADAto0-US-N}, there
exist two constants $M_1,\mu_1>0$ such that the solution $(\e,\e_t,Y,
\beta,\beta_t,\widetilde{w},\widetilde{w}_t)\in
C(0,\infty;\Xscr)$ satisfies
\begin{equation} \label{proof-epYtztw-to0-US-N}
   \|(\e(\cdot,t),\e_t(\cdot,t),Y(\cdot,t),\beta(\cdot,t),
   \beta_t(\cdot,t),\widetilde{w}(\cdot,t),\widetilde{w}_t
   (\cdot,t),Z(\cdot,t))\|_{\mathscr{X}}\leq M_1e^{-\mu_1 t}.
\end{equation}
Since $\widehat{w}(x,t)=\widetilde{w}(x,t)-Z(x,t)$ and $\widehat{w}_t
(x,t)=\widetilde{w}_t(x,t)+Z_x(x,t)$, we have that
\begin{equation} \label{proof-hatw-to0-US-N}
   \begin{array}{l} \|(\widehat{w}(\cdot,t),\widehat{w}_t(\cdot,t))
   \|_\hline \leq \|(\widetilde{w}(\cdot,t),\widetilde{w}_t(\cdot,t))
   \|_\hline + \|(Z(\cdot,t),Z_x(\cdot,t))\|_\hline \leq 3M_1
   e^{-\mu_1 t}. \end{array}
\end{equation}
Since $w(x,t)=\widehat{w}(x,t)-\e(x,t)$, $w_t(x,t)=\widehat{w}_t(x,t)
-\e_t(x,t)$, we obtain
\begin{equation} \label{proof-tw-to0-US-N}
   \|(w(\cdot,t),w_t(\cdot,t))\|_\hline \leq \|(\widehat{w}(\cdot,t),
   \widehat{w}_t(\cdot,t))\|_\hline + \|(\e(\cdot,t),\e_t(\cdot,t))
   \|_\hline \leq 4M_1 e^{-\mu_1 t} \m.
\end{equation}
It follows from \dref{proof-epYtztw-to0-US-N},
\dref{proof-hatw-to0-US-N} and \dref{proof-tw-to0-US-N} that
\dref{thm-whatwY-decay-US-N} holds with $M=6M_1$ and $\mu=\mu_1$.

Now we consider the ``$(\widehat{v},W)$-part'' which reads as
\begin{equation}\label{closed-equiv-hatv-W-part-US-N}
\left\{\begin{array}{l}
\widehat{v}_{tt}(x,t)=\widehat{v}_{xx}(x,t),  \crr\disp
\widehat{v}_x(0,t)=-q\widehat{v}_t(0,t)+c_1\widehat{v}(0,t),  \crr\disp
\widehat{v}_x(1,t)=-f(w(\cdot,t),w_t(\cdot,t))-d(t)-W_x(1,t),\crr\disp
W_t(x,t)=-W_x(x,t),\;\;
W(0,t)=-c_0\widehat{v}(0,t).
\end{array}\right.
\end{equation}
Since $f:H^1(0,1)\times L^2(0,1)\to\rline$ is continuous and
$(w,w_t)$ is bounded (since it tends to zero), $f(w(\cdot,t),w_t
(\cdot,t))\in L^\infty[0,\infty)$. Since $d\in L^\infty[0,\infty)$ or
$d\in L^2[0,\infty)$, it follows from Lemma \ref{Lem-hatz-bd-US-N}
that system \dref{closed-equiv-hatv-W-part-US-N} admits a unique
bounded solution, i.e.,
\begin{equation} \label{th-hatv-W-bd-US-N}
   \sup_{t\geq 0} \|(\widehat{v}(\cdot,t),\widehat{v}_t(\cdot,t),
   W(\cdot,t))\|_{\hline\times H^1(0,1)} < \infty \m.
\end{equation}
Noting that $W_t(x,t)=-W_x(x,t)$, it follows from \dref{hatv-def},
\dref{tildez-def} and \dref{th-hatv-W-bd} that
$$ \begin{array}{rl} \|(v(\cdot,t),v_t(\cdot,t))\|_\hline \nm &\leq\m
   \|(\widehat{v}(\cdot,t),\widehat{v}_t(\cdot,t))\|_\hline + \|
   (w(\cdot,t),w_t(\cdot,t))\|_{\hline},\crr\disp \|(z(\cdot,t),z_t
   (\cdot,t))\|_{\hline} \nm &\leq\m \|(\beta(\cdot,t),
   \beta_t(\cdot,t))\|_{\hline} + \|(\widehat{v}(\cdot,t),
   \widehat{v}_t(\cdot,t))\|_{\hline} + \|(W(\cdot,t),W_x(\cdot,t))
   \|_{\hline} ,\end{array}$$
which gives \dref{thm-vzW-bd-US-N}, because both right-hand sides are
bounded.

Now suppose that $f(0,0)=0$ and $d\in L^2[0,\infty)$. By
\dref{proof-tw-to0-US-N} and the continuity of $f$, we have
$\lim_{t\to\infty}|f(w,w_t)|=0$. By Lemma \ref{Lem-hatz-bd-US-N}, we
obtain
\begin{equation} \label{lem-v-W-bd0-th-US-N}
   \lim_{t\to\infty}\|(\widehat{v}(\cdot,t),\widehat{v}_t(\cdot,t),
   W(\cdot,t))\|_{\hline\times H^1(0,1)} \m=\m 0 \m.
\end{equation}
By \dref{proof-epYtztw-to0-US-N}, \dref{proof-tw-to0-US-N},
\dref{th-hatv-W-bd-US-N} and \dref{lem-v-W-bd0-th-US-N},
we derive
$$ \begin{array}{l} \|(v(\cdot,t),v_t(\cdot,t))\|_{\hline} \leq \|
   (\widehat{v}(\cdot,t),\widehat{v}_t(\cdot,t))\|_{\hline}
   + \|(w(\cdot,t),w_t(\cdot,t))\|_{\hline} \to 0,\mbox{ as } t\to
   \infty,\crr\disp \|(z(\cdot,t),z_t(\cdot,t))\|_{\hline} \leq \|
   (\widetilde{z}(\cdot,t),\widetilde{z}_t(\cdot,t))\|_{\hline} +
   \|(\widehat{v}(\cdot,t),\widehat{v}_t(\cdot,t))\|_{\hline} +
   \|(W(\cdot,t),W_x(\cdot,t))\|_{\mathbb{H}} \m,\end{array}$$
which is bounded. Next, suppose that $f\equiv 0$ and $d\equiv 0$.
It follows from Lemma \ref{Lem-hatz-bd-US-N} that there exist two
constants $M_2,\mu_2>0$ such that for all $t\geq 0$,
\begin{equation} \label{lem-v-W-bd-decay-th-US-N}
   \|(\widehat{v}(\cdot,t),\widehat{v}_t(\cdot,t),W(\cdot,t))
   \|_{\hline\times H^1(0,1)} \m\leq\m M_2e^{-\mu_2 t} \m.
\end{equation}
By \dref{proof-epYtztw-to0-US-N}, \dref{proof-tw-to0-US-N} and
\dref{lem-v-W-bd-decay-th-US-N}, we obtain that for all $t\geq 0$,
$$ \|(v(\cdot,t),v_t(\cdot,t))\|_{\mathbb{H}} \leq \|(\widehat{v}
   (\cdot,t),\widehat{v}_t(\cdot,t))\|_{\mathbb{H}} + \|(w(\cdot,t),
   w_t(\cdot,t))\|_{\mathbb{H}}
   \leq M_2e^{-\mu_2 t} + 4M_1e^{-\mu_1 t},$$
$$ \begin{array}{rl} \|(z(\cdot,t),z_t(\cdot,t))\|_{\mathbb{H}}
   &\leq\m \|(\beta(\cdot,t),\beta_t(\cdot,t))
   \|_{\mathbb{H}} + \|(\widehat{v}(\cdot,t),\widehat{v}_t(\cdot,t))
   \|_{\mathbb{H}} + \|(W(\cdot,t),W_x(\cdot,t))\|_{\mathbb{H}}\crr
   \disp &\leq\m M_1e^{-\mu_1 t}+2M_2e^{-\mu_2 t},\end{array}$$
which, combining with \dref{lem-v-W-bd-decay-th-US-N}, implies that
\dref{thm-vzW-bd-decay-US-N} holds.
\end{proof}

\begin{remark} \label{rek-antistablewave} {\rm
Similarly to Remark \ref{rek-unstablewave}, we point out that the
output measurement signals $w(0,t),w(1,t)$ are also almost a minimal
set of measurement signals to exponentially stabilize the system
\dref{wave-o-US-N}. Theorem \ref{Thm-closed-USN} shows that we can
design disturbance estimator and state observer by using
$\{w(0,t),w(1,t)\}$ only and that the system \dref{wave-o-US-N} can be
exponentially stabilized by using $\{w(0,t),w(1,t)\}$ only. However,\\
(a). Each of the observation $\{w(0,t),w(1,t)\}$ is not enough for
exact observability, i.e., for any $T>0$, there is no constant $C_T$
such that
$$ \int_0^T w^2(0,s) \dd s \m\geq\m C_T\|w(\cdot,0),w_t(\cdot,0)
   \|_{\mathbb{H}},\;\;\int_0^T w^2(1,s)\dd s \m\geq\m C_T\|w(\cdot,
   0),w_t(\cdot,0)\|_{\mathbb{H}}.$$
(b). The  $y(t)=w(1,t)$ is also  not enough for exponential
stabilizability. Actually, let $f(w)\equiv 0$, $d(t)=\mu e^{i\mu t}$
and $\phi(x)=\sin\mu(x-1)$, where $\mu$ satisfies $\cosh i\mu=q\sinh
i\mu$. Then, system \dref{wave-o-US-N} admits a solution $(w,w_t)=
(e^{i\mu t}\phi(x),i\mu e^{i\mu t}\phi(x))$ which makes the output
$y(t)=w(1,t)\equiv 0$.

From (a), (b), $w(0,t)$ seems to be necessary to ensure the
possibility of stabilization.  We leave two open question here: (I):
Can we design a state observer for system \dref{wave-o-US-N} using
only $y(t)=w(0,t)$? (II): Is $y(t)=w(0,t)$ only enough to make system
\dref{wave-o-US-N} stabilizable?}
\end{remark}

\section{Concluding remarks}

We have studied the exponential stabilization problem for the one
dimensional unstable or anti-stable wave equation with Neumann
boundary control subject to an unknown bounded disturbance, using
only two measurement signals. We have designed disturbance estimators
that do not use high gain and, based on these, have proposed state
observers. We have shown that the total disturbance is estimated by
the disturbance estimator in the sense that the error is in $L^2[0,
\infty)$, and that the state of the original system is recovered by
the proposed state observer. We have constructed a state observer
based output feedback controller that guarantees that the signals in
the original system are exponentially stable. This is a first output
feedback controller that can exponentially stabilize a system
described by PDEs with both internal uncertainty and external
disturbance. This shows that exponential stability can be achieved
without sliding mode control even for a very general type of
disturbance. Our approach can be generalized to deal with other PDEs
such as unstable/anti-stable wave equation with Dirichlet boundary
control matched with the internal uncertainty and the external
disturbance, again using two measurement signals. We have posed open
questions in Remarks \ref{rek-unstablewave} and
\ref{rek-antistablewave} concerning a stabilizing controller using
only one output measurement signal.


\section{Appendix}

{\bf The proof of Proposition \ref{Pro01}.} As already noted, $\aline$
is skew-adjoint, so that by the theorem of Stone, it generates a
unitary group of operators on $\hline$. In addition, it is not
difficult to show that $\bline_1$ and $\bline_2$ are admissible
control operators for $e^{{\mathbb A}t}$ - the details of all this are
in \cite[Example 5.2]{Nata_etal:16}. Therefore, for any fixed $T>0$,
and for any given $u,d\in L^2_{loc}[0,\infty)$, we have \vspace{-4mm}
\begin{equation}
   \m\ \ \ \int_0^t e^{\aline(t-s)}\bline_2[u(s)+d(s)] \dd s \m\in\m
   C(0,T;\hline).
\end{equation}

For any fixed $T>0$, we define on $C(0,T;\hline)$ the norm
\begin{equation} \label{norm-def}
   \|(\phi,\psi)\|_* \m=\m \sup_{t\in[0,T]} e^{-\l t}\|(\phi(\cdot,t),
   \psi(\cdot,t))\|_{\mathbb{H}} \m,
\end{equation}
where $\l$ is a positive constant to be determined later. It is
obvious that $C(0,T;\hline)$ with $\|\cdot\|_*$ is a Banach space.
Define the nonlinear map $\fline$ from $C(0,T;\hline)$ to
$C(0,T;\hline)$ by \vspace{-2mm}
$$ \fline \bbm{\varphi\\ \psi}(t) \m=\m e^{\aline t}
   \bbm{w_0\\ w_1} + \int_0^t e^{\aline(t-s)}\bline_2
   [u(s)+d(s)] \dd s \hspace{30mm} \m$$
\begin{equation} \label{F-map-def}
   \m\ \ \ -\int_0^t e^{\aline(t-s)}\bline_1((q+1) \varphi(0,s))\dd s
   + \int_0^t e^{\aline(t-s)}\bline_2 f(\varphi(\cdot,s),\psi
   (\cdot,s)) \dd s \m.
\end{equation}
We show that $\fline$ is a strict contraction on $C(0,T;\hline)$.
Indeed, since $f$ satisfies global Lipschitz condition in $H^1(0,1)
\times L^2(0,1)$, there exists a constant $L>0$ such that
\begin{equation} \label{f-Lip-con-E}
   |f(\phi_1,\psi_1)-f(\phi_2,\psi_2)| \m\leq\m L\|(\phi_1,\psi_1)
   -(\phi_2,\psi_2)\|_{\hline}.
\end{equation}
The admissibility of $\bline_1$ implies that for all $t>0$, there
exists $C_{1t}>0$ such that
$$ \left\|\int_0^t e^{\aline(t-s)}\bline_1(\varphi_1(0,s)-
   \varphi_2(0,s)) \dd s\right\|_{\hline}^2 \leq C_{1t}\|
   \varphi_1(0,s)-\varphi_2(0,s)\|_{L^2[0,t]}^2 \m.$$
From \cite[Proposition 2.3]{art01} we know that $C_{1t}$
is nondecreasing in $t$, hence $C_{1t}\leq C_{1T}$ for any $t\in
[0,T]$. It is easy to see from \dref{Eliel} that $|\varphi(0)|\leq
\|(\varphi,\psi)\|_\hline$ holds for all $(\varphi,\psi)\in\hline$.
Thus, for any $(\varphi_1,\psi_1),\m (\varphi_2,\psi_2)\in C(0,T;
\hline)$, \vspace{-2mm}
\begin{equation} \label{admiss-pro-T}
   \left\|\int_0^te^{\aline(t-s)} \bline_1(\varphi_1(0,s)-\varphi_2
   (0,s)) \dd s\right\|_\hline^2 \leq C_{1T} \int_0^t \left\|\sbm
   {\varphi_1(\cdot,s)\\ \psi_1(\cdot,s)} - \sbm{\varphi_2(\cdot,s)\\
   \psi_2(\cdot,s)}\right\|_{\hline} \dd s \m.
\end{equation}
Similarly to the above, by \dref{f-Lip-con-E}, the admissibility of
$\bline_2$ implies that for all $t>0$, \vspace{-2mm}
$$ \begin{array}{l}\disp \left\|\int_0^t e^{\aline(t-s)}\bline_2\left[
   f(\sbm{\varphi_1(\cdot,s)\\ \psi_1(\cdot,s)})-f(\sbm{\varphi_2
   (\cdot,s)\\ \psi_2(\cdot,s)} \right] \dd s\right\|_\hline^2\crr
   \disp \m\ \leq C_{2T}\left\|f(\sbm{\varphi_1(\cdot,s)\\ \psi_1
   (\cdot,s)})-f(\sbm{\varphi_2(\cdot,s)\\ \psi_2(\cdot,s)})\right\|
   _{L^2[0,t]}^2 \leq C_{2T}L^2\int_0^t \left\|\sbm{\varphi_1(\cdot,s)
   \\ \psi_1(\cdot,s)}-\sbm{\varphi_2(\cdot,s)\\ \psi_2(\cdot,s)}
   \right\|_\hline^2 \dd s \m.\end{array}$$
It follows from here and \dref{F-map-def}, \dref{admiss-pro-T} that
for any $(\varphi_1,\psi_1),(\varphi_2,\psi_2)\in C(0,T;\hline)$,
$$ \begin{array}{l}\disp \left\|\fline\sbm{\varphi_1\\ \psi_1}(t) -
   \fline\sbm{\varphi_2\\ \psi_2}(t)\right\|_\hline^2 \leq \left(
   C_{1T}(q+1)^2 + C_{2T} L^2\right) \int_0^t \left\|\sbm{\varphi_1
   (\cdot,s)\\ \psi_1(\cdot,s)} - \sbm{\varphi_2(\cdot,s)\\ \psi_2
   (\cdot,s)}\right\|_\hline^2 \dd s\crr\disp \m\ \ = \left(C_{1T}
   (q+1)^2+C_{2T}L^2 \right) \int_0^t e^{2\l s}e^{-2\l s}\left\|\sbm
   {\varphi_1(\cdot,s)\\ \psi_1(\cdot,s)} - \sbm{\varphi_2(\cdot,s)\\
   \psi_2(\cdot,s)}\right\|_\hline^2 \dd s\crr\disp \m\ \ \leq \left(
   C_{1T}(q+1)^2 + C_{2T}L^2\right) \m \frac{e^{2\l t}-1}{2\l}\left
   \|\sbm{\varphi_1\m\\ \psi_1\sbluff}-\sbm{\varphi_2\m\\ \psi_2
   \sbluff} \right\|_*^2 \FORALL t\in [0,T] \m. \end{array}$$
Choose $\l>\half[C_{1T}(q+1)^2+C_{2T}L^2]$ in \dref{norm-def}, then
the above estimate implies that $\fline$ is a strict contraction on
$C(0,T;\hline)$. By the contraction mapping theorem (see, for
instance, \cite{BrSc:09}), \dref{F-map-def} has a unique fixed point
$(\phi,\psi)\in C(0,T;\hline)$, which is then a solution of
\dref{Trump} in $[0,T]$, which implies that $\psi=\varphi_t$. Since
the above reasoning works for any $T>0$, \dref{wave-o} admits a unique
global solution. \hfill $\Box$

\end{document}